\documentclass[10pt,a4paper,oneside,reqno]{myart}
\usepackage[body={147mm,225mm}, footskip=9mm]{geometry}

\usepackage{amsthm,amssymb,amsmath}
\usepackage[english]{babel}
\usepackage{mathrsfs}
\usepackage{color}
\usepackage{esint}
\usepackage{url}
\usepackage{tikz}
\usepackage{subfigure}
\usepackage{caption}
\usepackage{verbatim} %fuer "comment"
 \usepackage{epstopdf}

 \usepackage{fancyhdr}
\let\oldmarginpar\marginpar\renewcommand\marginpar[1]{\-\oldmarginpar[\raggedleft\footnotesize #1]{\raggedright\footnotesize #1}}

\newtheorem{trm}{Theorem}[section]
\newtheorem{lem}[trm]{Lemma}
\newtheorem{prop}[trm]{Proposition}
\newtheorem{cor}[trm]{Corollary}
\theoremstyle{definition}

\newtheorem{rem}[trm]{Remark}

\newcommand{\Rho}{\mathrm{P}\!}

\newcommand{\x}{\mathbf{x}}
\newcommand{\kk}{\mathbf{k}}
\newcommand{\bfxi}{\boldsymbol{\xi}}
\newcommand{\argmin}{\operatorname{argmin}}

\newcommand{\sw}{\mathscr S}

\newcommand{\ran}{\operatorname{ran}}
\newcommand{\la}{\langle}
\newcommand{\ra}{\rangle}

\renewcommand{\d}{\,\mathrm{d}}
\newcommand{\pd}{\partial}
\newcommand{\e}{\mathrm{e}}
\newcommand{\ii}{{\mathrm{i}}}
\renewcommand{\phi}{\varphi}
\renewcommand{\epsilon}{\varepsilon}
\renewcommand{\theta}{\vartheta}
\newcommand{\C}{\mathbb{C}}
\newcommand{\R}{\mathbb{R}}

\newcommand{\cl}{\operatorname{cl}}
\newcommand{\aq}{\Leftrightarrow}

\newcommand{\Z}{\mathbb{Z}}
\newcommand{\N}{\mathbb{N}}
\renewcommand{\L}{\mathcal L}
\newcommand{\E}{\mathcal E}
\newcommand{\F}{\mathcal F}

\newcommand{\BB}{\mathscr B}
\newcommand{\CC}{\mathscr C}

\newcommand{\X}{\mathcal X}

\renewcommand{\P}{\mathcal P}

\newcommand{\spn}{\operatorname{span}}
\newcommand{\inj}{\hookrightarrow}
\renewcommand{\Im}{\operatorname{Im}}
\renewcommand{\Re}{\operatorname{Re}}

\newcommand{\diff}[3][]{\frac{\mathrm{d}^{#1}{#2}}{\mathrm{d}{#3}^{#1}}}
\newcommand{\pdiff}[3][]{\frac{\partial^{#1}{#2}}{\partial{#3}^{#1}}}
\newcommand{\nn}{|{}\hspace{-0.5mm}{}|{}\hspace{-0.5mm}{}|}

\renewcommand{\descriptionlabel}[1]{\hspace{\labelsep}\textbf{#1}:}

\numberwithin{equation}{section}
\allowdisplaybreaks

\begin{document}

\title[]{Spectral Analysis and Long-Time Behaviour of a Fokker-Planck Equation with a Non-Local Perturbation}
\begin{abstract}
In this article we consider a Fokker-Planck equation on $\R^d$ with a non-local, mass preserving perturbation. We first give a spectral analysis of the unperturbed Fokker-Planck operator in an exponentially weighted $L^2$-space. In this space the perturbed Fokker-Planck operator is an isospectral deformation of the Fokker-Planck operator, i.e.~the spectrum of the Fokker-Planck operator is not changed by the perturbation. In particular, there still exists a unique (normalized) stationary solution of the perturbed evolution equation. Moreover, the perturbed Fokker-Planck operator generates a strongly continuous semigroup of bounded operators. Any solution of the perturbed equation converges towards the stationary state with exponential rate $-1$, the same rate as for the unperturbed Fokker-Planck equation. Moreover, for any $k\in\N$ there exists an invariant subspace with codimension $k$ (if $d=1$) in which the exponential decay rate of the semigroup equals $-k$.
\end{abstract}

\author[D. St\"urzer]{Dominik St\"urzer} \address{Institute for Analysis and
 Scientific Computing, Technical University Vienna, Wiedner
 Hauptstra\ss{}e 8, A-1040 Vienna}
\email{dominik.stuerzer@tuwien.ac.at}

\author[A. Arnold]{Anton Arnold} \address{Institute for Analysis and
 Scientific Computing, Technical University Vienna, Wiedner
 Hauptstra\ss{}e 8, A-1040 Vienna}
\email{anton.arnold@tuwien.ac.at}

\subjclass[2010]{35B20, 35P99, 35Q84, 47D06}
\keywords{Fokker-Planck, non-local perturbation, spectral analysis, strongly continuous semigroup, stationary solution, long-time behaviour, exponential decay of the semigroup}

\maketitle

\section{Introduction}\label{sec1}
This work deals with the analysis of the following class of perturbed Fokker-Planck equations:
\begin{subequations}\label{pert_fp}
      \begin{eqnarray}
      \pd_t f&=&\nabla\cdot(\nabla f+\x f)+\Theta f=:Lf+\Theta f\label{pert_fp:1}\\
      f|_{t=0}&=&\phi(\x),\label{pert_fp:2}
      \end{eqnarray}
\end{subequations}
where $t\ge 0,\,\x\in\R^d$ with $d\in\N$, and $f=f(t,\x)$. Here, $\pd_t f$  denotes the time derivative. The linear, non-local operator $\Theta$ is given by a convolution $\Theta f=\theta*f$ with respect to $\x$, where its kernel $\theta$ is assumed to be time-independent and with zero mean, i.e.~$\int_{\R^d}\theta(\x)\d \x=0$. Also, it is assumed to satisfy certain regularity conditions, which will be specified in the Sections \ref{sec3} and \ref{sec35}.

The above equation is mainly motivated by the quantum-kinetic Wigner-\linebreak Fokker-Planck equation, describing so-called open quantum systems, see \cite{Arnold2010,alms}. It is of the form
\begin{align*}
 \pd_t u&= \nabla_{\x,\mathbf v}\cdot(\nabla_{\x,\mathbf v} u+ (\nabla_{\x,\mathbf v} A+\mathbf{F})u)+\Xi[V]u\\
 u|_{t=0}&=u_0,
\end{align*}
where $u=u(t,\x,\mathbf v)$ is the phase-space quasi-density, with $\x,\mathbf v\in\R^d$ denoting position and momentum. The given coefficient function $\nabla_{\x,\mathbf v} A+\mathbf{F}$ is affine in $(\x,\mathbf v)$ and models the confinement and friction of the system. $\Xi[V]$ is a non-local operator (convolution in $\mathbf v$) determined by an external potential $V(\x)$. One question of interest in this problem is to show the existence of a unique normalized stationary state, and to prove uniform exponential convergence of the solution to the stationary state. In the case of a quadratic confinement potential with a small perturbation these questions have been answered positively in \cite{Arnold2010}, see also \cite{afn2008} for an operator-theoretic approach. However, from the physical point of view, the restriction to nearly quadratic potentials seems quite artificial. This raises the question if the results can be extended to a more general family of (confining) potentials. In order to gain insight 
into what can be expected and what mechanisms are responsible for the actual behaviour, we shall consider here (\ref{pert_fp}) as a similar, yet simplified model, which still preserves the essential structure. The non-local operator $\Xi[V]$, which is a convolution in $\mathbf v$, is replaced by a convolution with kernel $\theta$. This represents a first step towards the full analysis.

Other examples of non-local perturbations in Fokker-Planck equations appear e.g.~in the linearized vorticity formulation of the 2D Navier-Stokes equations (cf.~(12)-(14) in \cite{gw}) or in electronic transport models (cf.~the linearization of equations (1), (6), (7) in \cite{lk}).

\medskip
For the unperturbed equation (\ref{pert_fp}), i.e.~the case $\theta=0$, the natural functional setting is the space $L^2(\mu^{-1})$, with the weight function $\mu(\x)=\exp(-|\x|^2/2)$. Here, $\mu/(2\pi)^{d/2}$ is the unique steady state with normalized mass, i.e.~$\int_{\R^d}\mu/(2\pi)^{d/2}\d \x=1$, and all solutions to initial conditions with mass one decay towards this state with exponential rate of at least $-1$, see e.g.~\cite{bakry}. However, if $\Theta$ is added, the situation often becomes more complicated. One reason is that many non-local (convolution) operators are unbounded in the space $L^2(\mu^{-1})$. This can be illustrated for the simple example with the convolution kernel $\theta=\delta_{-\alpha}-\delta_{\alpha}, \,\alpha\in\R$, in one dimension. It corresponds to the operator $(\Theta f)(x)=f(x+\alpha)-f(x-\alpha),\,x\in\R$, which is unbounded in $L^2(\mu^{-1})$. In this case one can show (with an eigenfunction expansion) that every (non-trivial) stationary state of (\ref{pert_fp}) is {\em not} even an element of $L^2(\mu^{-1})$. Thus, this space is not suitable for our intended large-time analysis, since it is ``too small''. This motivates to consider (\ref{pert_fp}) in some larger space $L^2(\omega)$, with a weight $\omega$ growing slower than $\mu^{-1}$. Due to the previous discussion we shall choose $\omega$ such that a large class of non-local operators becomes bounded. But the new space should not be ``too large'' either, since we would risk to loose many convenient properties (like the spectral gap) of the unperturbed Fokker-Planck operator. In $L^2(\R^d)$, e.g.,~the spectrum of $L$ is the left half plane $\{\lambda\in\C:\Re\lambda\le d/2\}$, cf.~\cite{Metafune2001}.  It will turn out that $\omega(\x):=\cosh \beta |\x|,\,\beta>0$, is a convenient choice. Moreover, there is a useful characterization of the functions of $L^2(\omega)$ in terms of their Fourier transform, see Lemma \ref{analyticity}.

Here we focus on the Fokker-Planck operator in exponentially weighted spaces. For $L^2$-spaces with polynomial weights, the spectrum of $L$ was studied in \cite{Gallay2002}. Furthermore, our results complement the analysis of Metafune \cite{Metafune2001}, where a larger class of Ornstein-Uhlenbeck operators is investigated in unweighted $L^p$-spaces with $p\ge 1$.

\medskip
This paper is organized as follows. Since the analysis in the $d$-dimensional case is very similar to the one-dimensional case, we first discuss (in Sections \ref{sec2} and \ref{sec3}) the one-dimensional problem in great detail, to keep the notation and arguments more concise. In Section \ref{sec35}, we generalize the proofs to higher dimensions.

In Section \ref{sec2} we investigate the one-dimensional Fokker-Planck operator in $L^2(\omega)$ (denoted by $\L$), and show that its spectrum is $-\N_0$, and consists entirely of eigenvalues. All eigenspaces are one-dimensional, in particular the stationary state is unique up to normalization. Moreover, the operator $\L$ generates a $C_0$-semigroup of uniformly bounded operators on $L^2(\omega)$, and any solution of (\ref{pert_fp}) for $\Theta=0$ converges towards the (appropriately scaled) stationary solution with exponential rate of at least $-1$. More generally, for any $k\in\N_0$ there exists an $\L$-invariant subspace of $L^2(\omega)$ with codimension $k$ in which the associated semigroup has an exponential decay rate of $-k$. Section \ref{sec3} is dedicated to the perturbed Fokker-Planck operator $\L+\Theta$ in one dimension. Using the compactness of the resolvent of $\L$ and ladder operators we show that $\L+\Theta$ is an isospectral deformation of the unperturbed operator $\L$, i.e.~$\sigma(\L+\Theta)=\sigma(\L)=-\N_0$. The spectrum still consists only of eigenvalues with one-dimensional eigenspaces, which ensures the existence of a unique normalized steady state of (\ref{pert_fp}) in $L^2(\omega)$. On a formal level this isospectral property of $\L+\Theta$ can be understood as follows: In the eigenbasis of $\L$, $\Theta$ corresponds to a strictly lower triangular (infinite) matrix. Finally we show that the semigroup generated by $\L+\Theta$ still has the same decay properties as the one generated by $\L$. In particular the solutions of (\ref{pert_fp}) with normalized mass decay to the stationary state with exponential rate of at least $-1$. In Section \ref{sec4} we present simulation results, which illustrate the decay rates obtained before.

\section{The Fokker-Planck Operator in Weighted $L^2$-Spaces}\label{sec2}

Here and in Section \ref{sec3} we shall consider the one-dimensional Fokker-Planck equation, i.e.~$d=1$. For the Fourier transform we use the convention
\[\F_{x\to\xi}f\equiv\hat f(\xi):=\int_\R f(x)\e^{-\ii x\xi}\d x.\]
With this scaling we may identify $\hat f(0)$ with the {\em mass} of $f$.

For an analytic function $f$ on a simply connected domain $\Omega$ we denote the line integral of $f$ along a path from $a$ to $b$ inside of $\Omega$ by
\[\int_{a\to b} f(\zeta)\d\zeta.\]
In order to properly define complex powers, we specify a branch of the logarithm. For $\xi\in\C\setminus\{0\}$ we set $\ln\xi:=\log|\xi|+\ii\arg\xi$, with $\arg\xi\in[-\frac\pi2,\frac{3\pi}2)$, and $\log(\cdot)$ is the natural logarithm on $\R^+$. For $\zeta\in\C$ we may then define $\xi^{\zeta}:=\exp(\zeta\ln(\xi))$.\label{log}

On a domain $\Omega\subseteq\R$ we call a real-valued function $w\in L^1_{\text{loc}}(\Omega)$ a {\em weight function} if it is bounded from below by a positive constant a.e.~on every compact subset of $\Omega$. We denote the corresponding weighted $L^p$-space by $L^p(\Omega;w)\equiv L^p(\Omega;w(x)\d x)$, where $1\le p\le \infty$. The space $L^2(\Omega;w)$ is equipped with the inner product \[\la f,g\ra_{\Omega,w}=\int_\Omega f\bar g w\d x,\] and the norm $\|\cdot\|_{\Omega,w}$.

Also, we introduce weighted Sobolev spaces. For two weight functions $w_0$ and $w_1$ and $1\le p\le\infty$, the space $W^{1,p}(\Omega;w_0,w_1)$ consists of all functions $f\in L^p(\Omega;w_0)$, whose distributional derivative satisfies $f'\in L^p(\Omega;w_1)$. We equip the space $W^{1,2}(\Omega;w_0,w_1)$ with the norm
\[
      \|f\|_{\Omega,w_0,w_1}:=\big(\|f\|_{\Omega,w_0}^2+\|f'\|_{\Omega,w_1}^2 \big)^{\frac12},
\]
see \cite{Kufner1984}. If $\Omega=\R$ we shall omit the symbol $\Omega$ in these notations.

\medskip

Furthermore, we present some definitions and properties concerning unbounded operators and their spectrum. Let $X,\X$ be Hilbert spaces. If $X$ is continuously and densely embedded in $\X$, we write $X\inj \X$, and $X\inj\inj \X$ indicates that the embedding is compact. $\CC(X)$ denotes the set of all closed operators $A$ in $X$ with dense domain $D(A)$. The set of all bounded operators $A:X\to \X$ is $\BB(X,\X)$; if $X=\X$ we just write $\BB(X)$. A closed, linear subspace $Y\subset X$ is said to be {\em invariant} under $A\in\CC(X)$ (or {\em $A$-invariant}) iff $D(A)\cap Y$ is dense in $Y$ and $\ran A|_{Y}\subset Y$, see e.g.~\cite{Albrecht2003}. For an operator $A\in\CC(X)$ its range is $\ran A$, its null space is $\ker A$, and its algebraic null space is $M(A):=\bigcup_{k\ge 0}\ker A^k$. For any $\zeta\in\C$ lying in the resolvent set $\rho(A)$, we denote the resolvent by $R_A(\zeta):=(\zeta-A)^{-1}$. The complement of $\rho(A)$ is the spectrum $\sigma(A)$, and $\sigma_p(A)$ is the point spectrum. For an 
isolated subset $\sigma'\subset\sigma(A)$ the corresponding {\em spectral projection} $\Rho_{A, \sigma'}$ is defined via the line integral 
\begin{equation}\label{def:spec_proj}
 \Rho_{A, \sigma'}:=\frac 1{2\pi\ii}\oint_\Gamma R_A(\zeta)\d\zeta,
\end{equation}
where $\Gamma$ is a closed Jordan curve with counter-clockwise orientation, strictly separating $\sigma'$ from $\sigma(A)\setminus\sigma'$, with $\sigma'$ in the inside of $\Gamma$ and $\sigma(A)\setminus\sigma'$ on the outside. The following results can be found in \cite[Section III.6.4]{kato} and \cite[Section V.9]{taylay}: The spectral projection is a bounded projection operator, decomposing $X$ into two $A$-invariant subspaces, namely $\ran \Rho_{A,\sigma'}$ and $\ker \Rho_{A,\sigma'}$. This property is referred to as the {\em reduction of $A$ by $\Rho_{A, \sigma'}$.} A remarkable property of this decomposition is the fact that $\sigma(A|_{\ran \Rho_{A,\sigma'}})= \sigma'$ and $\sigma(A|_{\ker \Rho_{A, \sigma'}})=\sigma(A)\backslash \sigma'$. Most of the time we will be concerned with the situation where $\sigma'=\{\lambda\}$, i.e.~an isolated point of the spectrum. For further results see the Appendix \ref{app:space_enlarg}.

\smallskip
A final remark concerns constants occurring in estimates: Throughout this article, $C$ denotes some positive constant, not necessarily always the same. Dependence on certain parameters will be indicated in brackets, e.g.~$C(t)$ for dependence on $t$.
\medskip

We begin our analysis by investigating the unperturbed one-dimensional\linebreak  Fokker-Planck operator $Lf:=f''+xf'+f$ in various weighted spaces. The natural space to consider $L$ in is $E:=L^2(1/\mu)$ with $\mu(x):=\exp(-x^2/2)$. We use the notation $\|\cdot\|_E$ for the norm and $\la\cdot,\cdot\ra_E$ for the inner product. Writing the operator in the form \[Lf=\bigg(\bigg(\frac f\mu\bigg)'\mu\bigg)'\] shows that $L|_{C_0^\infty}$ is symmetric and dissipative in $E$. Then, the proper definition of $L$ is obtained by the closure of $L|_{C_0^\infty}$, and this procedure yields its domain $D(L)\subset E$. In the subsequent theorem we summarize some important properties of $L$ in $E$, see \cite{Metafune2001,bakry,helffernier}. Since $L$ in $E$ is isometrically equivalent to the (dimensionless) quantum harmonic oscillator Hamiltonian $H=-\Delta-1/2+x^2/4$ in $L^2(\R)$, we transfer many results of $H$ (see \cite{parmegg} and \cite[Theorem XIII.67]{resi}) to $L$. For the properties of the spectral projections, 
see also \cite[Section V.3.5]{kato}.

\begin{trm}\label{prop:fp_in_H}
The Fokker-Planck operator $L$ in $E$ has the following properties:
      \begin{enumerate}
      \renewcommand{\theenumi}{\roman{enumi}}
\renewcommand{\labelenumi}{(\theenumi)}
	\item\label{st_fp:1} $L$ with $D(L)=\{f\in E:f''+xf'+f\in E\}$ is self-adjoint and has a compact resolvent.
	\item\label{st_fp:2} The spectrum is $\sigma(L)=-\N_0$, and it consists only of eigenvalues.
	\item\label{st_fp:3} For each eigenvalue $-k\in\sigma(L)$ the corresponding eigenspace is one-di\-men\-sio\-nal, span\-ned by  $\mu_k:= \frac 1{\sqrt{2\pi}}H_k\mu$, where \[H_k(x)= \mu(x)^{-1}\diff[k]{}{x}\mu(x)\]
	is the $k$-th Hermite polynomial.
	\item\label{st_fp:4} The eigenvectors $(\mu_k)_{k\in\N_0}$ form an orthogonal basis of $E$.
	\item\label{st_fp:45} There holds the spectral representation
	      \[\label{spec_proj_in_E}L=\sum_{k\in\N_0}-k\Pi_{L,k},\quad\text{where}\quad \Pi_{L,k}:=\frac{\sqrt{2\pi}}{k!}\mu_k\la
	\cdot,\mu_k\ra_E\] is the spectral projection onto the $k$-th eigenspace.
	\item\label{st_fp:5} The operator $L$ generates a $C_0$-semigroup of contractions on $E_k$ for all $k\in\N_0$, where $E_k:=\ker(\Pi_{L,0}+\cdots+\Pi_{L,k-1}),\,k\ge 1$, and $E_0:=E$ are $L$-invariant subspaces of $E$. The semigroup satisfies the estimate
	      \[\big\|\e^{tL}|_{E_k}\big\|_{\BB(E_k)}\le\e^{-kt},\quad\forall k\in\N_0.\]
      \end{enumerate}
\end{trm}

Hence, the Fokker-Planck equation $\pd_t f=Lf$ has a unique stationary solution with normalized mass, given by $\mu_0$. Its orthogonal complement $E_1$ consists of all elements of $E$ with zero mass. And according to Result (\ref{st_fp:5}) for $k=1$, any solution of $\pd_t f=Lf$ with unit mass converges towards $\mu_0$ with exponential rate of at least $-1$ in the $E$-norm.

\medskip
In order to analyze the perturbed equation \eqref{pert_fp}, we quickly find that $E$ is not appropriate. For example, for the simple (unbounded) perturbation $\Theta f(x):=f(x+\alpha)-f(x-\alpha),\,\alpha\in\R$, we can explicitly compute the stationary solution $f_0$ of \eqref{pert_fp} and expand it with respect to the orthogonal basis $(\mu_k)_{k\in\N}$ of $E$. The obtained Fourier coefficients form a divergent sequence, and so $f_0\notin E$. Therefore we consider some larger space $L^2(\omega)$ instead of $E$, with a weight function $\omega$ growing more slowly than $\mu^{-1}$. Thereby we choose $\omega$ such that $\Theta$ becomes a bounded operator in $L^2(\omega)$ for a large family of convolution kernels.  E.g., one can easily verify that $\Theta f(x)=f(x+\alpha)-f(x-\alpha)$ is bounded in $L^2(\exp(\beta|x|^\gamma))$ iff $\gamma\in[0,1]$ (for $\beta>0$). At the same time, $\omega$ should grow fast enough such that $L$ still has a spectral gap in $L^2(\omega)$, i.e.~there exists some $a<0$ such that $\{\zeta\in\C:\Re\zeta>a\}\cap\sigma(L)=\{0\}$. These requirements suggest that exponentially growing weights would be good candidates, growing as fast as permissible while still admitting a large class of non-local operators. So, for the rest of this paper, we choose the weight function $\omega(x)=\cosh\beta x$ for some fixed $\beta>0$, and use the corresponding space $\E:=L^2(\cosh\beta x)$. As we will see in the following, the space $\E$ is very convenient also for technical purposes, since it can easily be characterized using the Fourier transform.

\begin{lem}\label{analyticity}
For $f\in\E$ we have the following properties:
\begin{enumerate}
\renewcommand{\theenumi}{\roman{enumi}}
\renewcommand{\labelenumi}{(\theenumi)}
 \item\label{analyt:i} There holds $f\in \E$ iff its Fourier transform $\hat f$ possesses an analytic continuation (still denoted by $\hat f$) to the open strip $\Omega_{\beta/2}:=\{z\in \C:|\Im z|<\beta/2\}$, which satisfies
\begin{equation}\label{seid}
 \sup_{\substack{|b|<\beta/2\\b\in\R}}\|\hat f(\cdot+\ii b)\|_{L^2(\R)}<\infty.
\end{equation}
 \item\label{analyt:ii} For $\xi\in\R$ and $|b|<\beta/2$, $\hat f$ is explicitly given by $\hat f(\xi+\ii b)=\F_{x\to\xi}(\e^{bx}f(x))$. 
 \item\label{analyt:iii} The following function lies in $L^2(\R)$:
\begin{equation}\label{hatf}
\xi\mapsto\hat f\Big(\xi\pm\ii\frac\beta 2 \Big):=\F_{x\to\xi}(\e^{\pm\frac\beta 2x}f(x)),\quad\text{for a.e.~}\xi\in\R.
\end{equation}
Moreover, $b\mapsto\hat f(\cdot+\ii b)$ lies in $C([-\beta/2,\beta/2];L^2(\R))$. In particular \eqref{hatf} is a natural continuation of $\hat f$ from $\Omega_{\beta/2}$ to the closure $\overline{\Omega_{\beta/2}}$.
\end{enumerate}
\end{lem}

The proof is deferred to the Appendix \ref{app:proof}. In the following, $\hat f$ always denotes the extension of the Fourier transform of $f\in\E$ according to Lemma \ref{analyticity} (\ref{analyt:ii})-(\ref{analyt:iii}). Using this convention, we introduce an alternative norm on the space $\E$:
\begin{equation}\label{norm_w}
 \nn f\nn_\omega^2:=\|\hat f(\cdot+\ii \beta/2)\|_{L^2(\R)}^2+\|\hat f(\cdot-\ii \beta/2)\|_{L^2(\R)}^2,
\end{equation}
which is equal to $4\pi\|f\|^2_\omega$. 

Furthermore, we notice that there holds a Poincar\'e-type inequality in $\E$:

\begin{lem}[Poincar\'e inequality]
The inequality
\begin{equation}\label{poincare}
 \| f\|_\omega\le C_\beta\| f'\|_\omega
\end{equation}
holds for all $f\in W^{1,2}(\omega,\omega)$, where $C_\beta>0$ is a constant only depending on $\beta$.
\end{lem}

\begin{proof}
Use $|\widehat {f'\,}\!(\xi)|=|\xi\hat f(\xi)|$, and $|\xi|\ge \beta/2$ on $|\Im\xi|=\beta/2$. Then apply the norm $\nn\cdot\nn_\omega$.
\end{proof}

Our next step is to properly define the Fokker-Planck operator in $\E$. To this end we first define the distributional Fokker-Planck operator $\mathfrak Lf:=f''+xf'+f$ for $f\in\sw'$.

\begin{lem}\label{reso_estim}
Let $\zeta\in\C$ with $\Re \zeta\ge 1+\beta^2/2$, and consider the resolvent equation $(\zeta-\mathfrak L)f=g$ for $f,g\in\E$. Then there exists a constant $C>0$ independent of $f,g$, such that
\begin{equation}\label{comp_est}
  \|f\|_\varpi+\|f'\|_\omega \le C\|g\|_\omega,
\end{equation}
where $\varpi(x)=(1+|x|)\omega(x)$.
\end{lem}

\begin{proof}
Let us fix $\zeta\in\C$ with $\Re \zeta\ge 1+\beta^2/2$. Now we consider the resolvent equation $(\zeta-\mathfrak L)f=g$ for $f,g\in \E\subset\sw'$. Applying $\la\cdot,f\ra_\omega$ to both sides yields:
\begin{align*}
 \int_\R \bar f g\omega\d x&= \int_\R \zeta |f|^2\omega-(f'+xf)'\bar f\omega\d x\\
&= \int_\R |f'|^2\omega+|f|^2(x\omega'+\zeta\omega)+f'\bar f\omega' +f\bar f' x\omega\d x. 
\end{align*}
Next we take the real part:
\begin{align}
\Re \int_\R \bar f g\omega\d x&= \int_\R |f'|^2\omega+|f|^2(x\omega'+\Re(\zeta)\omega)+\frac 12|f^2|'(\omega' +x\omega)\d x\nonumber\\
      &=  \int_\R |f'|^2\omega+\frac 12|f|^2\tilde\omega\d x,\label{4star}
\end{align}
with $\tilde\omega:=-\omega''+x\omega'+(2\Re\zeta-1)\omega$.
For our choice $\omega(x)=\cosh \beta x$ we obtain $\tilde\omega(x)=(2\Re\zeta-1-\beta^2)\omega(x)+x\beta \sinh\beta x$. For $\Re\zeta\ge 1+\beta^2/2$, $\tilde\omega$ is strictly positive. Thus, $\tilde\omega$ is a weight function, and it has the asymptotic behaviour $\tilde\omega(x)\sim\beta|x|\omega(x)$ as $x\to\pm\infty$. Applying the Cauchy-Schwarz inequality to the left hand side of (\ref{4star}) yields
\[
\frac12\|f\|_{\tilde\omega}^2+ \|f'\|^2_\omega\le \|f\|_\omega\|g\|_\omega.
\]
For the left hand side we use $\omega(x)\le \tilde\omega(x)$ and the Poincar\'e inequality (\ref{poincare}) to obtain
\[
 \frac 12\|f\|_{\tilde\omega}+\frac 1{C_\beta}\|f'\|_\omega \le \|g\|_\omega.
\]
The result follows, since the weight functions $\tilde\omega$ and $\varpi$ define equivalent norms.
\end{proof}

\begin{cor}\label{cor_dissip}
 The operator $(L-1-\beta^2/2)|_{C_0^\infty(\R)}$ is dissipative in $\E$.
\end{cor}

\begin{proof}
 We use the result \eqref{4star} for $\zeta= 1+\beta^2/2$. We then estimate the right hand side for $f\in C_0^\infty(\R)$:
\[\Re\int_\R\bar f(L-\zeta)f\d x \le -\Big(C_\beta+\frac 12\Big)\|f\|_\omega^2\le 0,\]
where we used the Poincar\'e inequality and $\tilde\omega\ge\omega$.
\end{proof}

The above results can be used to establish the proper definition of the Fokker-Planck operator in $\E$:

\begin{lem}\label{abschluss_LL}
 The operator $L|_{C_0^\infty(\R)}$ is closable in $\E$. Its closure $\L:=\cl_\E L|_{C_0^\infty(\R)}$ has the domain of definition $D(\L)=\{f\in\E: \mathfrak Lf\in\E\}$. For $f\in D(\L)$ we have $\L f=\mathfrak Lf$.
\end{lem}

The proof is deferred to the Appendix \ref{app:proof}. It also yields the following result:

\begin{cor}\label{wob}
 The resolvent set $\rho(\L)$ is non-empty. It contains the half-plane $\{\zeta\in\C:\Re \zeta\ge 1+\beta^2/2\}$.
\end{cor}

As it turns out, the resolvent estimate \eqref{comp_est} is strong enough to prove compactness of the resolvent. To this end we shall use the following simplified version of \cite[Theorem 2.4]{Opic1989}:

\begin{lem}
Let $w,w_0,w_1$ be weight functions, and $(\Omega_n)_{n\in\N}$ a monotonically increasing sequence of
subsets of $\R$ that converges to $\R$. Assume that for all $n\in\N$ there holds the compact embedding $W^{1,2}(\Omega_n;w_0,w_1)\inj\inj L^2(\Omega_n; w)$. Then
\[
W^{1,2}(w_0,w_1)\inj\inj L^2(w)\quad\aq\quad \lim_{n\to\infty}\sup_{\|f\|_{w_0,w_1}\le 1}\|f\|_{\R\backslash\Omega_n;w}=0.
\]
\end{lem}

From this we deduce immediately the following lemma:

\begin{lem}\label{comp_embed}
Let $w,w_0,w_1$ be weight functions. If $\lim_{|x|\to\infty}w(x)/w_0(x)=0$, then the compact embedding holds:
\[W^{1,2}(w_0,w_1)\inj\inj L^2(w).\]
\end{lem}

This compact embedding allows to prove that $R_\L(\zeta)$ is compact:
\begin{trm}\label{trm:compactness}
For any $\zeta \in\rho(\L)$ the resolvent operator $R_\L(\zeta):\E\to\E$ is compact. In particular $\sigma(\L)=\sigma_p(\L)$, i.e.~the spectrum of $\L$ consists entirely of eigenvalues.
\end{trm}

\begin{proof}
To begin with, we fix some $\zeta\in\C$ with $\Re \zeta\ge 1+\beta^2/2$. According to Lemma \ref{reso_estim} we have the estimate \eqref{comp_est}, which we can reformulate: There exists a constant $C>0$ such that 
\[\|R_\L(\zeta)g\|_{\varpi,\omega}\le  C\|g\|_\omega,\quad \forall g\in\E.\]
Hence $R_\L(\zeta)\in\BB(\E, W^{1,2}(\varpi,\omega))$. Now there holds the asymptotic behaviour $\omega(x)/\varpi(x)\allowbreak\sim 1/|x|\to 0$ as $x\to\pm\infty.$ Therefore we may apply Lemma \ref{comp_embed} for $w=w_1=\omega$ and $w_0=\varpi$, which yields the compact embedding $W^{1,2}(\varpi,\omega)\inj\inj\E$. Thus, the resolvent $R_\L(\zeta)\in\BB(\E)$ is compact for $\Re \zeta\ge 1+\beta^2/2$. But this already implies the compactness of $R_\L(\zeta)$ for all $\zeta\in\rho(\L)$, cf.~\cite[Theorem III.6.29]{kato}. The same reference confirms that $\sigma(\L)=\sigma_p(\L)$.
\end{proof}

With these preparations we can now characterize the spectrum of $\L$: 

\begin{prop}\label{prop_spec}
 We have $\sigma(\L)=-\N_0$. Each eigenspace is one-dimensional, and for $k\in\N_0$ we have $\ker(k+\L)=\spn\{\mu_k\}$.
\end{prop}

\begin{proof}
We consider the Fourier transform of the eigenvalue equation $(\zeta-\mathfrak L)f=0$ for $f\in \E$. The general solution of the Fourier-transformed equation on the real line reads:
\begin{equation}\label{branchis}
 \hat f(\xi)=C_\pm\mu(\xi)\xi^{-\zeta},\quad \xi\in\R^\pm.
\end{equation}
For details see the computation in the beginning of the Appendix \ref{app:b} for $g=\theta=0$. Since $f\in\E$, $\hat f$ has to be analytic in $\Omega_{\beta/2}$, see Lemma \ref{analyticity}. With the specification of the complex logarithm in Section \ref{sec2} we may extend both parts of $\hat f$ from \eqref{branchis} analytically to the complex half-planes $\{\Re\xi>0\}$ and $\{\Re\xi<0\}$ respectively. However, if $\zeta\in\C\setminus\Z$, the two extensions do not meet continuously at the imaginary axis, thus $\hat f$ is not analytic in $\Omega_{\beta/2}$ (except for the trivial case $C_\pm=0$). If $\zeta\in\Z$, we obtain continuity of $\hat f$ at the imaginary axis (without $\xi=0$) iff $C_-=C_+$. But for $\zeta\in\N$, $\hat f$ still has a pole at $\xi=0$, thus it is not analytic. In the remaining case $\zeta\in -\N_0$ the function $\hat f$ from \eqref{branchis} has an analytic extension to $\C$, when we choose $C_-=C_+$. So $f\in\E$ solves the eigenvalue equation for $\zeta$ iff $\zeta\in-\N_0$. And 
according to \eqref{branchis} the eigenspaces are still spanned by the $\mu_k,\,k\in\N_0$, since $\hat\mu_k(\xi)=(\ii\xi)^k\mu(\xi)$.
\end{proof}

The main difference to $L$ in $E$ is that the eigenfunctions do not form an orthogonal basis any more. However, we are still able to transfer the concept of the $L$-invariant subspaces $E_k\subset E$ to $\E$.

\begin{samepage}
\begin{prop}\label{spec_pr}
For every $k\in\N$ we have the following facts:
\begin{enumerate}\renewcommand{\theenumi}{\roman{enumi}}
\renewcommand{\labelenumi}{(\theenumi)}
 \item\label{ho:i} The subspace $\E_k:=\cl_\E E_k$ is $\L$-invariant, and $\sigma(\L|_{\E_k})=\{-k,-k-1,\ldots\}$
 \item\label{ho:ii} The spectral projection $\Pi_{\L,k}$ of $\L$ associated to the eigenvalue $-k$ satisfies
\[
 \ker\Pi_{\L,k}=\E_{k+1}\oplus\spn\{\mu_{k-1},\ldots,\mu_0\},\quad\ran\Pi_{\L,k}= \spn\{\mu_k\}.
\]
Moreover, $\ker\Pi_{\L,0}=\E_1$ and $\ran\Pi_{\L,0}=\spn\{\mu_0\}$.
 \item\label{ho:iii} There holds $\E=\E_k\oplus \spn\{\mu_{k-1},\ldots,\mu_0\}$.
\end{enumerate}
\end{prop}
\end{samepage}

\begin{proof}
Since $\sigma(L)=\sigma(\L)$, and $R_L(\zeta)\subset R_\L(\zeta)$ for all $\zeta\in\C\setminus(-\N_0)$, we conclude from \eqref{def:spec_proj} that for any $\sigma'\subset\sigma(\L)$ there holds $\Pi_{L,\sigma'}\subset \Pi_{\L,\sigma'}$, and they are bounded projections in $E$ and $\E$, respectively. For $\sigma':=\{0,\ldots,-k+1\},\, k\in\N,$ we apply Lemma \ref{lem:proj} from the appendix: $\ran\Pi_{\L,\sigma'}=\cl_\E\ran\Pi_{L,\sigma'}=\cl_\E\spn\{\mu_0,\ldots,\mu_{k-1}\}\allowbreak=\spn\{\mu_0,\ldots,\mu_{k-1}\}$ and $\ker\Pi_{\L,\sigma'}=\cl_\E\ker\Pi_{L,\sigma'}=\cl_\E E_k=:\E_k$. This shows \eqref{ho:i}. Since the projection $\Pi_{\L,\sigma'}$ is bounded, the range and kernel indeed represent a decomposition of $\E$, thus we also obtain Result \eqref{ho:iii}.

For \eqref{ho:ii} we use the same arguments as before, with $\sigma'=\{-k\}$ instead.
\end{proof}

Next we characterize the subspaces $\E_k$.

 \begin{prop}\label{char:e_k}
For $k\in-\N$ the subspace $\E_k$ is explicitly given by
\begin{equation}\label{e_k} \E_k=\left\{f\in\E:\int_\R f(x)x^j\d x=0,\, 0\le j\le k-1\right\}.\end{equation}
Furthermore, there holds
\begin{equation}\label{f_char_e_k}\E_k=\left\{f\in\E:\hat f^{(j)}(0)=0,\,0\le j\le k-1\right\},\end{equation}
where $\hat f^{(j)}$ denotes the $j$-th derivative of the Fourier transform of $f$.
\end{prop}

\begin{proof}
The functionals $\psi_j:f\mapsto\int_\R f(x)x^j\d x,\,j\in\N$, are continuous in $\E$. We define $\tilde \psi_j:=\psi_j|_{E}$. Let $f\in E_k=\{\mu_0,\ldots,\mu_{k-1}\}^{\perp_E}$. The orthogonality condition then reads 
\[0=\la f,\mu_j\ra_E=\int_\R f(x)\mu_j(x)\mu(x)^{-1}\d x=\frac 1{\sqrt{2\pi}}\int_\R f(x)H_j(x)\d x,\quad\forall 0\le j\le k-1,\]
which is equivalent to $\tilde\psi_0(f)=\ldots=\tilde\psi_{k-1}(f)=0$. Applying Lemma \ref{lem:funct} from the appendix with $\X=\E$ and $X=E$ yields $\cl_\E E_k=\{f\in\E:\psi_j(f)=0,\,0\le j\le k-1\}$, which is equal to $\E_k$ by definition. This proves (\ref{e_k}).

The second equality (\ref{f_char_e_k}) immediately follows from
\[\int_\R f(x)x^j\d x=\F_{x\to\xi}[f(x)x^j](0)=\ii^j\hat f^{(j)}(0),\quad\forall j\in\N_0.\]
\end{proof}

\begin{rem}
 The representation (\ref{e_k}) of the $\E_k$ also holds in polynomially weighted spaces, which is shown in \cite[Appendix A]{Gallay2002}.
\end{rem}

The final result of this section deals with the analysis of the semigroup $(\e^{t\L})_{t\ge 0}$ generated by $\L$ in $\E$. We already know that $L$ generates a $C_0$-semigroup $(\e^{tL})_{t\ge 0}$ of bounded operators in $E$, and from \cite[Appendix A]{Gallay2002} we get its representation (for $f\in E$):
\begin{equation}\label{expl_semi}
 \F_{x\to\xi}\big[\e^{tL}f\big]=\exp\Big(-\frac{\xi^2}2(1-\e^{-2t})\Big)\hat f\big(\xi\e^{-t}\big),\quad t\ge0.
\end{equation}
This formula can be extended to $f\in \E$, yielding a family $(S(t))_{t\ge 0}$ of operators in $\E$.

\begin{lem}\label{cx_0}
 The family of operators $(S(t))_{t\ge0}$ given by \eqref{expl_semi} is a family of bounded operators in $\E$.
\end{lem}

\begin{proof}
In order to show that the operators $S(t)$ are bounded, we use the norm $\nn\cdot\nn_\omega$. So we estimate $\|\F[ S(t)f](\xi+\ii\beta/2)\|$, the estimate for the other term in $\nn\cdot\nn_\omega$ is analogous:
\begin{align}
 \big\|\F[ S(t)f](\cdot+\ii\beta/2)\big\|^2_{L^2(\R)} \!& =\! \int_\R \!\!\exp\Big(\Big[-\xi^2+\frac{\beta^2}4\Big](1-\e^{-2t})\Big)\Big|\hat f\Big(\Big[\xi+\ii\frac\beta 2\Big]\e^{-t}\Big)\Big|^2\d\xi\label{zw_nn}\\
&\le \exp\Big(\frac{\beta^2}4\Big)\int_\R \Big|\hat f\Big(\Big[\xi+\ii\frac\beta 2\Big]\e^{-t}\Big)\Big|^2\d\xi\nonumber\\
&=\exp\Big(\frac{\beta^2}4+t\Big)\int_\R \Big|\hat f\Big(\xi+\ii\e^{-t}\frac\beta 2\Big)\Big|^2\d\xi\nonumber\\
&\le \exp\Big(\frac{\beta^2}4+t\Big)\nn f\nn_{\cosh(\e^{-t}\beta x)}^2    \le  \exp\Big(\frac{\beta^2}4+t\Big)\nn f\nn^2_\omega\nonumber
\end{align}
So $( S(t))_{t\ge0}$ is a family of bounded operators in $\E$, and there exists a constant $M>0$ with
\[
\| S(t)\|_{\BB(\E)}\le M\e^{t/2},\quad t\ge0.
\]
\end{proof}

\begin{lem}
 The operator $\L$ is the infinitesimal generator of the $C_0$-semigroup $(S(t))_{t\ge0}$ in $\E$.
\end{lem}

\begin{proof}
According to \cite[Theorem 1.4.5]{pazy}, Corollary \ref{cor_dissip} implies that $\L-1-\beta^2/2=\cl_\E(L|_{C_0^\infty}-1-\beta^2/2)$ is dissipative in $\E$. From Proposition \ref{prop_spec} we also know that any $\zeta\in\C$ with $\Re\zeta>0$ lies in $\rho(\L)$. So we can apply the Lumer-Phillips Theorem \cite[Theorem 1.4.3]{pazy} and find that $\L$ generates a $C_0$-semigroup $(\e^{t\L})_{t\ge 0}$ of bounded operators. Since $\e^{t\L}$ and $S(t)$ are both bounded in $\E$ and coincide on the dense subspace $D(L)\subset\E$, we get $\e^{t\L}=S(t)$ in $\E$ for all $t\ge 0$.
\end{proof}

As a consequence we write $\e^{t\L}:=S(t)$ for the semigroup generated by $\L$, and the representation \eqref{expl_semi} holds for all $f\in\E$.

\begin{prop}\label{unpert_decay}
For every $k\in\N_0$ we have:
\begin{enumerate}
\renewcommand{\theenumi}{\roman{enumi}}
\renewcommand{\labelenumi}{(\theenumi)}
 \item\label{1i} The space $\E_k$ is invariant under the family  $(\e^{t\L})_{t\ge0}$.
 \item\label{2ii}  There exists some $C_k>0$ such that
\[
 \|\e^{t\L}|_{\E_k}\|_{\BB(\E_k)}\le C_k\e^{-kt},\quad t\ge 0.
\]
\end{enumerate}
\end{prop}

\begin{proof}
The closed subspaces $\E_k$ are $\L$-invariant, so they are also invariant under $(\e^{t\L})_{t\ge0 }$.

In order to show \eqref{2ii}, we use the first line of \eqref{zw_nn} and make the additional assumption $t\ge1$:
\begin{align}
 \big\|\F[\e^{t\L}f](\xi+\ii\beta/2)\big\|^2_{L^2(\R_\xi)} & \le \e^{\frac{\beta^2}4} \int_\R \e^{-\frac{\xi^2}2}\Big|\Big[\xi+\ii\frac\beta 2\Big]\e^{-t}\Big|^{2k}\left|\frac{\hat f\Big(\Big[\xi+\ii\frac\beta 2\Big]\e^{-t}\Big)}{\Big(\Big[\xi+\ii\frac\beta 2\Big]\e^{-t}\Big)^k}\right|^2\d\xi\label{hello}
\end{align}
Here we used the inequality $\frac 12<1-\e^{-2t}<1$ for $t\ge1$. In the following we use the Poincar\'e inequality \eqref{poincare}:
\begin{align*}
 \left\|\frac{\hat f\Big(\Big[\xi+\ii\frac\beta 2\Big]\e^{-t}\Big)}{\Big(\Big[\xi+\ii\frac\beta 2\Big]\e^{-t}\Big)^k}\right\|_{L^\infty(\R_\xi)}&=\left\|\F_{x\to\xi}\left(\exp\Big(\frac\beta2\e^{-t}x\Big)\F^{-1}_{\xi\to x}\left[\frac{\hat f(\xi)}{\xi^k}\right]\right)\right\|_{L^\infty(\R_\xi)} \\
&\le\left\|\exp\Big(\frac\beta2\e^{-t}x\Big)\F^{-1}_{\xi\to x}\left[\frac{\hat f(\xi)}{\xi^k}\right]\right\|_{L^1(\R_x)}\\
&\le \tilde C(t)\left\|\F^{-1}_{\xi\to x}\left[\frac{\hat f(\xi)}{\xi^k}\right]\right\|_\omega\\
&\le  C(t)\left\|(\ii \pd_x)^k\F^{-1}_{\xi\to x}\left[\frac{\hat f(\xi)}{\xi^k}\right]\right\|_\omega =  C(t)\|f\|_\omega.
\end{align*}
Thereby, the constant $\tilde C(t)$ is given by 
\[\tilde C(t)=\int_\R \frac{\exp(\beta\e^{-t}x)}{\cosh\beta x}\d x,\]
which is uniformly bounded for $t\ge 1$. Inserting this result in \eqref{hello} yields for $t\ge1$
\begin{align*}
 \big\|\F[\e^{t\L}f](\xi+\ii\beta/2)\big\|^2_{L^2(\R_\xi)}&\le C\e^{\frac{\beta^2}4}\e^{-2kt}\|f\|^2_\omega  \int_\R \e^{-\frac{\xi^2}2}\Big|\xi+\ii\frac\beta 2\Big|^{2k}\d \xi\\
&= C \e^{-2kt}\|f\|^2_\omega.
\end{align*}
Thus there exists a constant $C>0$ such that $\nn \e^{t\L}f\nn_\omega\le C\e^{-kt}\nn f\nn_\omega$ for all $t\ge 1$. From Lemma \ref{cx_0} we also know that the semigroup is uniformly bounded for $t\in[0,1]$, so altogether we get the desired decay estimate for the semigroup in $\E_k$.
\end{proof}

Before we turn to the perturbed Fokker-Planck equation, we summarize our results so far:

\begin{trm}\label{extension_theorem}
Let $\omega(x):=\cosh\beta x$ for some $\beta>0$. Then the Fokker-Planck operator $L|_{C_0^\infty(\R)}$ is closable in $\E=L^2(\omega)$, and its closure $\L=\cl_\E L|_{C_0^\infty(\R)}$ has the following properties:
\begin{enumerate}
\renewcommand{\theenumi}{\roman{enumi}}
\renewcommand{\labelenumi}{(\theenumi)}
 \item\label{fp:res:i} The spectrum satisfies $\sigma(\L)=-\N_0$, and $\ker(\L+k)=\spn\{\mu_k\}$ for any $k\in\N_0$. The eigenfunctions satisfy the relation $\mu_k=\mu_0^{(k)}$, the $k$-th derivative of $\mu_0$.
 \item The resolvent $R_\L(\zeta)$ is compact in $\E$ for all $\zeta\notin -\N_0$.
 \item\label{fp:res:ii} For any $k\in\N_0$ the closed subspace $\E_k:=\cl_\E\spn\{\mu_k,\mu_{k+1},\ldots\}$ is an $\L$-invariant subspace of $\E$, and $\spn\{\mu_0,\ldots,\mu_{k-1}\}$ is a complement. In particular $\E_0=\E$.
\item\label{fp:res:iii} The spectral projection $\Pi_{\L,k}$ to the eigenvalue $-k\in-\N_0$ fulfills $\ran\Pi_{\L,k}=\spn\{\mu_k\}$ and $\ker\Pi_{\L,k}=\E_{k+1}\oplus \spn\{\mu_{k-1},\ldots,\mu_0\}$ for $k\in\N_0$.
 \item\label{fp:res:iv} For any $k\in\N_0$ the operator $\L$ generates a $C_0$-semigroup on $\E_k$, and there exists a
constant $C_k\ge 1$ such that we have the estimate
\[
 \left\|\e^{t\L}|_{\E_k}\right\|_{\BB(\E_k)}\le C_k \e^{-kt},\quad\forall t\ge 0.
\]
\end{enumerate}
\end{trm}

\begin{rem}
More generally, the results of Theorem \ref{extension_theorem} hold for all weight functions $\omega(x)=\exp(\beta|x|^\gamma)$ with either $\gamma\in(0,2)$ and $\beta>0$ or  $\gamma=2$ and $\beta\in(0,\frac 12]$. This can be shown by using the results from \cite{gmm}, where an operator decomposition method is used to transfer spectral properties of operators from a Banach space to a larger Banach space. For a detailed discussion of the application of \cite{gmm}, see \cite{meidiss}.
\end{rem}

\begin{rem}
 The sequence of eigenfunctions $(\mu_k)_{k\in\N_0}$ is an orthogonal basis of $E$. In the larger space $\E$, the linear hull $\spn\{\mu_k:k\in\N_0\}$ is still dense, due to the continuous embedding $E\inj\E$.

Also, each $f\in\E$ can (formally) uniquely be decomposed according to the sequence of spectral projections $(\Pi_{\L,k})_{k\in\N_0}$, see the proof of Proposition \ref{pert:spec}. But the obtained series may diverge in $\E$. As an example we consider $f(x):=\exp(-|x|)\in L^2(\cosh x)$. Since $f$ is symmetric, we have $\Pi_{\L,k}f=0$ if $k$ is odd. For $k=2n,\,n\in\N_0$, one can show the asymptotic behaviour for $n\to\infty$:
 \[\|\Pi_{\L,2n}f\|_\omega=\mathcal{O}\Big(\frac{\sqrt{(2n)!}}{n^{1/4}}\Big),\]
 where we use the explicit representation for the Hermite polynomials $H_{2n}$ from (5.5.4) in \cite{szego}, and the asymptotic expansions for $H_{2n}$ given in \cite[Theorem 8.22.9]{szego}. Therefore, the formal series $\sum_{n\in\N_0}\Pi_{\L,2n} f$ is divergent in $\E$. So the sequence $(\mu_k)_{k\in\N_0}$ is neither a Schauder basis nor a representation system of $\E$. However, the sequence $(\mu_k/\|\mu_k\|_E)_{k\in\N_0}$ is still a Bessel system, see \cite{christensen, bilgus} for the definitions.
\end{rem}

\section{Analysis of the Perturbed Operator}\label{sec3}

So far we have discussed the one-dimensional Fokker-Planck operator $\L$ in $\E=L^2(\omega)$, with $\omega(x)=\cosh\beta x$. In this section we investigate the properties of the perturbed (one-dimensional) operator $\L+\Theta$ in $\E$, and we shall summarize the results in Theorem \ref{final_pert_trm}. We begin by specifying the assumptions we make on the perturbation $\Theta$.

\medskip
\begin{samepage}\label{pagge}
\noindent{\bf (C) Conditions on $\boldsymbol\Theta$:} We assume that $\Theta f=\theta*f$, for $f\in\E$, where $\theta$ is a tempered distribution that fulfills the following properties in $\Omega_{\beta/2}$ for some $\beta>0$:
\begin{enumerate}
\renewcommand{\theenumi}{\roman{enumi}}
\renewcommand{\labelenumi}{(\theenumi)}
	\item The Fourier transform $\hat \theta$ can be extended to an analytic function in $\Omega_{\beta/2}$ (also denoted by $\hat \theta$), and $\hat\theta\in L^\infty(\Omega_{\beta/2})$.
	\item It holds $\hat\theta(0)=0$, i.e.~$\theta$ has zero mean.
	\item The mapping $\xi\mapsto\Re\int_0^1\hat\theta(\xi s)/s\d s$ is essentially bounded in $\Omega_{\beta/2}$.
\end{enumerate}
\end{samepage}

\begin{rem}\label{rem31}
 If the conditions {\bf (C)(i)-(ii)} hold for $\theta$, then the mapping $\xi\mapsto\int_0^1\hat\theta(\xi s)/s\d s$ is analytic in $\Omega_{\beta/2}$. This becomes clear when writing $\hat\theta(\xi s)/s= \xi\hat \theta(\xi s)/(\xi s)$, which is analytic for all $s\in(0,1]$ and can be continuously extended to $\hat\theta'(0)\xi$ for $s=0$. The analyticity of $\xi\mapsto\int_0^1\hat\theta(\xi s)/s\d s$ on $\Omega_{\beta/2}$ then follows from \cite[Theorem 4.9.1]{dettman}.
\end{rem}

\begin{lem}\label{thetaeine}
There holds $\Theta f\in\E$ for all $f\in\E$ iff the condition {\bf (C)(i)} holds.  
\end{lem}

\begin{proof}
Clearly, $\widehat{\Theta f}=\hat\theta\hat f$ is analytic in $\Omega_{\beta/2}$ for $f\in\E$. According to Lemma \ref{analyticity} there holds $\Theta f\in\E$ iff 
      \begin{equation}\label{theta_in_e}
	\sup_{|b|<\beta/2}\|(\hat\theta\hat f) (\cdot+\ii b)\|_{L^2(\R)}<\infty,
      \end{equation}
where we use $\widehat {\Theta f}=\hat\theta\hat f$. Now we apply H\"older's inequality and find that (\ref{theta_in_e}) holds for all $f\in\E$ iff $\theta$ satisfies {\bf (C)(i)}.
\end{proof}

As a consequence of the above lemma and \eqref{theta_in_e}, the product $\hat\theta\hat f$ itself is the Fourier transform of an element of $\E$. So we may define $(\hat\theta\hat f)(\cdot\pm\ii\beta/2)\in L^2(\R)$ for $f\in\E$ according to (\ref{hatf}) whenever $\theta$ satisfies {\bf (C)(i)}. With this we obtain according to Lemma \ref{analyticity} (\ref{analyt:iii}):
\begin{equation}\label{fthet}
 b\mapsto (\hat\theta\hat f)(\cdot+\ii b)\in C([-\beta/2,\beta/2];L^2(\R)). 
\end{equation}

\begin{cor}\label{lemma:bdd_theta}
The convolution $\Theta$ is bounded in $\E$ if the condition {\bf (C)(i)} holds.
\end{cor}

\begin{proof}
We apply the norm (\ref{norm_w}) to $\Theta f$.  The Fourier transform turns the convolution into a multiplication, so we get according to \eqref{fthet} and {\bf (C)(i)}
\begin{samepage}
 \begin{align*}
 \nn\Theta f\nn_\omega^2&=\int_\R |\hat\theta\hat f(\xi-\ii\beta/2)|^2\d \xi+\int_\R |\hat\theta\hat f(\xi+\ii\beta/2)|^2\d \xi\\
&=\lim_{b\nearrow\beta/2}\Big[\int_\R |\hat\theta\hat f(\xi-\ii b)|^2\d \xi+\int_\R |\hat\theta\hat f(\xi+\ii b)|^2\d \xi\Big]\\
&\le \|\hat\theta\|^2_{L^\infty(\Omega_{\beta/2})}\lim_{b\nearrow\beta/2}\Big[\int_\R |\hat f(\xi-\ii b)|^2\d \xi+\int_\R |\hat f(\xi+\ii b)|^2\d \xi\Big]\\
&=\|\hat\theta\|^2_{L^\infty(\Omega_{\beta/2})}\nn f\nn_\omega^2.
\end{align*}
\end{samepage}
\end{proof}

\begin{lem}\label{remark_theta_inv}
Under the assumption {\bf (C)} there holds $\Theta:\E_k\to\E_{k+1}\subset\E_k$ for every $k\in\N$.
\end{lem}

\begin{proof}
According to Proposition \ref{char:e_k}, $f\in\E_k$ iff $\xi=0$ is a zero of $\hat f(\xi)$ of order greater or equal to $k$. Because of the assumption $\hat \theta(0)=0$ the Fourier transform $\widehat{\Theta f}=\hat \theta\hat f$ has a zero at least of order $k+1$ for $f\in\E_k$, so $\Theta f \in\E_{k+1}$.
\end{proof}

\begin{cor}\label{cor:inv_subsp}
Let {\bf (C)} hold, and $k\in\N_0$. Then the space $\E_k$ is an $(\L+\Theta)$-invariant subspace of $\E$. 
\end{cor}

Since the conditions {\bf (C)} are not very handy for direct applications, the following lemma gives some criteria that are simpler to verify and sufficient for {\bf (C)}.

\begin{samepage}\begin{lem}\label{hinr_theta}
 Let $\beta>0$ and $\omega(x)=\cosh\beta x$, and assume that $\theta\in\sw'$ fulfills 
\begin{enumerate}
\renewcommand{\theenumi}{\roman{enumi}}
\renewcommand{\labelenumi}{(\theenumi)}
 \item\label{zeromas} $\hat \theta(0)=0$,
 \item $\theta=\theta_W+\theta_D$ with $\theta_W\in W^{1,1}(\omega^{\frac12},\omega^{\frac12})$ and  $\theta_D\in D:=\{\sum_{j=1}^n a_j\delta_{x_j}:a_j\in\C,\,x_j\in\R,\,n\in\N\}$, where $\delta_{x_j}$ denotes the delta distribution located at $x_j$.
\end{enumerate}
Then $\Theta f=\theta*f$ satisfies {\bf (C)} for this $\beta>0$.
\end{lem}
\end{samepage}

\begin{proof}
In general $\hat\theta_W(0)$ and $\hat\theta_D(0)$ are not zero, so it is convenient to define $\theta_W^*:=\theta_W+M\mu$ and $\theta_D^*:=\theta_D-M\mu$, where $M:=\hat\theta_D(0)/\sqrt{2\pi}$. Then $\hat\theta_W^*$ and $\theta_D^*$ have zero mass, and we still have $\theta_W^*\in W^{1,1}(\omega^{\frac12},\omega^{\frac12}).$ Since $\F_{x\to\xi} \delta_{x_j}=\e^{-\ii \xi x_j}$ and $\hat\mu(\xi)=\sqrt{2\pi}\mu(\xi)$, it is immediate that $\theta^*_D$ satisfies {\bf (C)(i)}. In order to see {\bf (C)(iii)} for $\theta^*_D$, we note that the integral occurring in this condition can be rewritten as the line integral from $0$ to $\xi$:
\[\int_{0\to\xi}\frac{\hat\theta_D^*(z)}{z}\d z\]
which is path-independent in $\C$ (and thus in $\Omega_{\beta/2}$), since $\hat\theta_D^*(z)/z$ is analytic in $\Omega_{\beta/2}$ with a removable singularity at $z=0$. Therefore the integral itself is analytic, and thus uniformly bounded on every compact subset of $\C$. Because of this, it is sufficient to show uniform boundedness of this integral as $|\xi|\to\infty$ in $\Omega_{\beta/2}$. We outline this for the map $\xi\mapsto\e^{-\ii x_j\xi}$ for any fixed $x_j\in\R$ and $\Re \xi>1$, the case $\Re\xi<-1$ is analogous. Thereby we choose the following integration path (note that we may start from $z=1$, since the integral from $0$ to $1$ is a constant)
\begin{align*}
 \left|\int_{1\to\xi}\frac{\e^{-\ii x_jz}}{z}\d z\right|& \le\left| \int_1^{\Re(\xi)}\frac{\e^{-\ii x_jz}}{z}\d z\right|+\left|\int_{\Re(\xi)\to\Re(\xi)+\ii\Im(\xi)}\frac{\e^{-\ii x_jz}}{z}\d z\right|\\
&\le\left| \int_{x_j}^{\Re(\xi)x_j}\frac{\e^{-\ii z}}{z}\d z\right|+\frac\beta 2\e^{|x_j|\beta/2}.
\end{align*}
The first integral is known to remain uniformly bounded as $\Re(\xi)\to+\infty$. For estimating the second integral we used $\xi\in\Omega_{\beta/2}$ and $\Re\xi\ge1$.
Since $\hat\mu=\sqrt{2\pi}\mu$ decays sufficiently fast in $\Omega_{\beta/2}$, it is clear that the integral of $\hat\mu(z)/z$ from $1$ to $\xi$ also remains uniformly bounded as $\xi\to+\infty$. Altogether, we conclude that $\hat\theta_D^*$ satisfies {\bf (C)(iii)}.

Now we verify the same properties for $\theta_W^*$. Since $\theta_W^*\in L^1(\omega^{\frac 12})$, we may extend $\hat \theta_W^*$ to an analytic function in $\Omega_{\beta/2}$, and there holds (\ref{hatf}), cf.~\cite[Proposition XVI.1.3]{dautli5}. The Fourier transform is a continuous map from $L^1(\R)$ to $B_0(\R)$, i.e.~the continuous functions decaying at infinity, equipped with the uniform norm. Therefore, $\theta_W^*\in L^1(\omega^{\frac 12})$ implies
\begin{align*}
  \|\hat \theta_W^*\|_{L^\infty(\Omega_{\beta/2})}&=\sup_{|b|<\frac\beta2}\sup_{\xi\in\R}|\hat\theta_W^*(\xi+\ii b)|
      \le \!\!\sup_{|b|<\frac\beta2}\|\theta_W^*(x)\e^{bx}\|_{L^1(\R)}\\
    &  \le \|\theta_W^*(x)\e^{\frac\beta 2|x|}\|_{L^1(\R)}<\infty.
\end{align*}

So {\bf (C)(i)} is satisfied. For {\bf (C)(iii)} it is sufficient to show that for some $c>0$ and all $\xi\in\Omega_{\beta/2}$ with $|\xi|\ge 1$ there holds $|\hat\theta_W^*(\xi)|\le c/|\xi|$, which is fulfilled if $\F({\theta_W^*}')\in\ L^\infty(\Omega_{\beta/2})$. Analogously to the previous part of the proof we obtain that this is satisfied if ${\theta_W^*}'\in L^1(\omega^{\frac 12})$. We conclude that $\theta_W^*$ fulfills {\bf (C)(i)} and {\bf (C)(iii)} if $\theta_W^*\in W^{1,1}(\omega^{\frac12},\omega^{\frac12})$.

Finally, $\theta$ satisfies the condition {\bf (C)(ii)} due to the assumption (\ref{zeromas}).
\end{proof}

For the rest of the article, we shall always assume that $\Theta$ satisfies the condition {\bf (C)} for some fixed $\beta>0$ , and we choose the weight function $\omega(x)=\cosh \beta x$ with this particular $\beta$. The first result about the perturbed Fokker-Planck operator is the following lemma:

\begin{lem}\label{sigma_perturb}
The operator $\L+\Theta$ has compact resolvent in $\E$.
\end{lem}

\begin{proof}
A bounded perturbation of an infinitesimal generator with compact resolvent has compact resolvent again, see
\cite[Proposition III.1.12]{engel}. Then the result follows by combining the results of Theorems \ref{trm:compactness} and \ref{extension_theorem} for $\L$, and Corollary \ref{lemma:bdd_theta} for $\Theta$. 
\end{proof}

As a consequence, the spectrum of $\L+\Theta$ in $\E$ is non-empty and consists only of eigenvalues. In order to characterize the entire spectrum, we introduce the following ladder operators\footnote{One of the best-known applications of ladder operators occurs in the spectral analysis of the quantum harmonic oscillator, see e.g.~\cite{helffer2002}.}, namely the {\em annihilation operator} 
\[\alpha^-:\E_1\to\E:f\mapsto \int_{-\infty}^x f(y)\d y,\]
and its formal inverse $\alpha^+:f\mapsto f'$, the {\em creation operator}.

\begin{lem}\label{prop_annil}
The annihilation operator $\alpha^-$ has the following properties:
\begin{enumerate}\renewcommand{\theenumi}{\roman{enumi}}
\renewcommand{\labelenumi}{(\theenumi)}
 \item\label{anihi:i} For any $k\in\N$ there holds $\alpha^-\in\BB(\E_k,\E_{k-1})$.
 \item\label{anihi:ii} In $\E_1$ the operators $\Theta$ and $\alpha^-$ commute.
 \item\label{anihi:iii} Let $f\in\E_1,\,\zeta\in\C$ such that $(\L+\Theta)f=\zeta f$. Then 
	\[(\L+\Theta)(\alpha^-f)=(\zeta+1) (\alpha^-f).\]
\end{enumerate}

\end{lem}

\begin{proof}
First we show (\ref{anihi:i}). The property $\alpha^-:\E_k\to\E_{k-1}$ can be verified by using the explicit representation (\ref{e_k}) of the $\E_k$, and integration by parts (first for $f\in C_0^\infty(\R)$). The boundedness of $\alpha^-$ follows immediately from the Poincar\'e inequality (\ref{poincare}). Property (\ref{anihi:ii}) holds true since $\Theta$ is a convolution. For Result (\ref{anihi:iii}) one applies $\alpha^-$ to the equation $(\L+\Theta)f=\zeta f$, and uses the identity $\alpha^-(\L f)=\L(\alpha^-f)-\alpha^-f$ and the Property (\ref{anihi:ii}).
\end{proof}

By using the annihilation operator, we are able to prove:

\begin{prop}\label{pert:spec}
We have the following spectral properties of $\L+\Theta$ in $\E$:
\begin{enumerate}\renewcommand{\theenumi}{\roman{enumi}}
\renewcommand{\labelenumi}{(\theenumi)}
 \item\label{ps:i}  $\sigma(\L+\Theta)=-\N_0$.
 \item\label{ps:ii} For each $k\in\N_0$, the eigenspace $\ker(\L+\Theta+k)$ is one-dimensional.% and spanned by the corresponding eigenfunction, denoted by $f_k$.
 \item\label{ps:iii} The eigenfunction $f_k$ to the eigenvalue $-k\in\-\N_0$ is explicitly given by (up to a normalization constant)
\begin{equation}\label{rec_f_k}
 f_k=(\alpha^+)^kf_0=f_0^{(k)},\quad\text{and}\quad \hat f_0(\xi)=\exp\Big(-\frac{\xi^2}2+\int_0^1\frac{\hat\theta(\xi s)}s\d s\Big),\quad\xi\in\Omega_{\beta/2}.
\end{equation}
\end{enumerate}
In particular, $f_0$ is the unique stationary solution with unit mass of the perturbed Fokker-Planck equation (\ref{pert_fp}) in one dimension.
\end{prop}

\begin{proof}
In order to show (\ref{ps:i}) we first prove that $\bigcap_{k\in\N}\E_k=\{0\}$. According to (\ref{f_char_e_k}) there holds
\[\bigcap_{k\in\N}\E_k=\left\{f\in\E:\hat f^{(k)}(0)=0,\, k\in\N_0\right\}.\]
But for $f\in\E$, $\hat f$ is analytic, and the only analytic function with a zero of infinite order is the zero function, which proves the statement.

Thus, for any eigenfunction $f$, there exists a unique $k\in\N_0$ such that $f\in\E_k\backslash\E_{k+1}$, which is the minimal $k\in\N_0$ with the property $\Pi_{\L,k} f\neq 0$. Applying this projection to the eigenvalue equation yields
\[\Pi_{\L,k}(\L+\Theta)f=-k\Pi_{\L,k}f=\zeta\Pi_{\L,k}f,\]
where we used $\Theta f\in\E_{k+1}$ (cf.~Lemma \ref{remark_theta_inv}). Hence, the eigenvalue corresponding to $f$ satisfies $\zeta=-k$. Thus  $\sigma(\L+\Theta)\subseteq-\N_0$. If now $f_k$ is an eigenfunction with eigenvalue $-k$, we can apply $k$ times the continuous operator $\alpha^-$ to $f_k$, and create eigenfunctions to all eigenvalues $\{-k+1,\ldots,0\}$. So either $\sigma(\L+\Theta)=-\N_0$ or $\sigma(\L+\Theta)=\{-k_0,\ldots,0\}$, i.e.~there exists some minimal eigenvalue $-k_0$. But the latter scenario is actually not possible, because then the operator $(\L+\Theta)|_{\E_{k_0+1}}$ would have empty spectrum in $\E_{k_0+1}$, which contradicts the fact that it still has a compact resolvent in $\E_{k_0+1}$.

In order to verify (\ref{ps:ii}) we recall from the first part of the proof that if $f$ is an eigenfunction of $\L+\Theta$ to the eigenvalue $-k$, then $k=\argmin\{\Pi_{\L,j}f\neq 0:j\in\N_0\}$. In particular,
\begin{equation}\label{3.6}
  \Pi_{\L,k}f\neq 0                                                                                                                                                     \end{equation}
for such an eigenfunction. Assume that $\dim\ker(\L+\Theta+k)>1$ for some $k\in\N_0.$ Thus we may choose two linearly independent eigenfunctions to the eigenvalue $-k$. Since $\dim\ran\Pi_{\L,k}=1$, we can find a linear combination of these two eigenfunctions, yielding an eigenfunction $f$ which satisfies $\Pi_{\L,k}f=0$. But this contradicts \eqref{3.6} and hence $\dim\ker(\L+\Theta+k)=1$.

For the third result (\ref{ps:iii}) we consider the Fourier transform of the eigenvalue equation $(\L+\Theta)f_k=-k f_k$ for $k\in\N_0$. This yields the following differential equation for $\hat f_k$:
\[\xi\hat f_k'(\xi)=\big(\hat \theta(\xi)+k-\xi^2\big)\hat f_k(\xi).\]
Its general solution reads
\[
 \hat f_k(\xi)=c_k\xi^k q(\xi),\quad\text{with}\quad q(\xi):=\exp\Big(-\frac{\xi^2}2+\int_0^1\frac{\hat\theta(\xi s)}s\d s\Big),
\]
for all $k\in\N_0$, with $c_k\in\C$. We may now fix $c_k:=\ii^k$, which completes the proof.
\end{proof}

\begin{rem}
 According to the results of Proposition \ref{spec_pr} \eqref{ho:ii} we may formally write $\Theta$ and $\L$ as infinite-dimensional matrices with respect to the eigenfunctions $\mu_k, k\in\N_0$. Due to the property $\Theta:\E_k\to\E_{k+1}$ shown in Lemma \ref{remark_theta_inv} this representation of $\Theta$ is {\em strictly lower triangular}. Furthermore, due to Theorem \ref{extension_theorem} (\ref{fp:res:ii}), $\L$ is formally diagonal. And according to Proposition \ref{spec_prop} $\sigma(\L)=\sigma(\L+\Theta)$. This situation resembles the finite-dimensional case, in which adding a strictly triangular matrix does not change the spectrum of a diagonal matrix.
\end{rem}

 \begin{lem}\label{def:spec:proj:pert}
The spectral projection $\mathcal P_k$ of $\L+\Theta$ corresponding to the eigenvalue $-k\in-\N_0$ fulfills \[\ran \mathcal P_k=\spn\{f_k\},\quad\ker\mathcal P_k=\E_{k+1}\oplus\spn\{f_{k-1},\ldots,f_0\},\]
with the eigenfunctions $f_k,\ldots,f_0$ given in (\ref{rec_f_k}). Therefore, all singularities of the resolvent are of order one, and for all $k\in\N_0$ there holds $M(\L+\Theta+k)=\ker(\L+\Theta+k)$. 
\end{lem}

\begin{proof}
The set $\mathcal{K}_k:=\E_{k+1}\oplus\spn\{f_{k-1},\ldots,f_0\}$ is invariant under $\L+\Theta$, cf.~Corollary \ref{cor:inv_subsp}. Therefore the algebraic eigenspace satisfies  $M(\L+\Theta+k)=\ker(\L+\Theta+k)=\spn\{f_k\}$, being the complement of $\mathcal K_k$. In particular we obtain the $(\L+\Theta)$-invariant decomposition $\E=\mathcal K_k\oplus M(\L+\Theta+k)$, and $\sigma((\L+\Theta)|_{\mathcal K_k})=-\N_0\backslash\{-k\}$. So we can apply Lemma \ref{uniqueness} from the appendix, which yields the properties of the spectral projections.

Since $\dim \mathcal P_k=1$ and $M(\L+\Theta+k)=\ker(\L+\Theta+k)$, the singularity of $R_{\L+\Theta}(\zeta)$ at $\zeta=-k$ is a pole of order one, see Proposition \ref{spec_prop} (\ref{itemiii})-(\ref{itemiv}).
\end{proof}

Having explicitly determined the spectrum of the perturbed Fokker-Planck operator, we now turn to the generated semigroup and the corresponding decay rates. We start with the fact that $\L+\Theta$ generates a $C_0$-semigroup:

\begin{prop}\label{prop_3_11}
For each $k\in\N_0$ the operator $(\L+\Theta)|_{\E_k}$ is the infinitesimal generator of a $C_0$-semigroup on $\E_k$. The
semigroup on $\E$ preserves mass, i.e.~
\[\int_\R f(x)\d x=\int_\R [\e^{t(\L+\Theta)}f](x)\d x,\quad\forall t\ge 0.\]
\end{prop}

\begin{proof}
According to Theorem \ref{extension_theorem} the operator $\L$ generates a $C_0$-semigroup on $\E_k$ for every $k\in\N_0$, and due to Lemma \ref{remark_theta_inv} and Corollary \ref{lemma:bdd_theta} we have $\Theta|_{\E_k}\in\BB(\E_k)$. Now a bounded perturbation of the infinitesimal generator of a $C_0$-semigroup is again infinitesimal generator, see \cite[Theorem III.1.3]{engel}, and so the first result follows.

To show the conservation of mass we use the decomposition of $(\e^{t(\L+\Theta)})_{t\ge 0}$ by $\mathcal P_0$ corresponding to $\E=\E_1\oplus\spn\{f_0\}$. The space $\E_1$ consists of all massless functions, so the part $\mathcal P_0f$ alone determines the mass of any $f\in\E$. Since $\E_1$ and $\spn\{ f_0\}$ are both invariant under the semigroup, $\mathcal P_0$ and $(\e^{t(\L+\Theta)})_{t\ge 0}$ commute. Furthermore we have $\mathcal P_0 f\in\ker(\L+\Theta)$, and hence $\e^{t(\L+\Theta)}\mathcal P_0 f=\mathcal P_0 f$ for all $t\ge 0$. Altogether we obtain $\mathcal P_0 \e^{t(\L+\Theta)}f=\mathcal P_0 f$ for all $f\in\E,\,t\ge 0$, i.e.~the semigroup preserves mass.
\end{proof}

Next we investigate the decay rate of $(\e^{t(\L+\Theta)})_{t\ge 0}$ on the subspaces $\E_k$. To this end we define:
\begin{equation}
\hat\psi(\xi):=\exp\Big(\int_0^1\frac{\hat\theta(\xi s)}s\d s\Big),\quad\xi\in\Omega_{\beta/2},\label{last_ref} 
\end{equation}

which is analytic in $\Omega_{\beta/2}$ according to Remark \ref{rem31}.
\begin{lem}\label{eigs:Psi}
The map $\Psi:f\mapsto f*\psi$ has the properties:
\begin{enumerate}\renewcommand{\theenumi}{\roman{enumi}}
\renewcommand{\labelenumi}{(\theenumi)}
 \item\label{klam_i} For each $k\in\N_0$, $\Psi:\E_k\to\E_k$ is a bijection, with inverse $\Psi^{-1}:f\mapsto f*\F^{-1}[1/\hat\psi]$.
 \item $\Psi,\Psi^{-1}\in\BB(\E)$.
\end{enumerate}
\end{lem}

\begin{proof}
We define $\bar\Psi:f\mapsto f*\F^{-1}[1/\hat\psi]$. Due to the condition {\bf (C)(iii)} there holds $\Psi f,\bar\Psi f\in \E$ for all $f\in\E$, which is shown analogously to Lemma \ref{thetaeine}. Let now $f\in\E_k$ for some $k\in\N_0$. Then $\hat f(\xi)$ has a zero of order greater or equal to $k$ at $\xi=0$, cf.~Proposition \ref{char:e_k}. Since $\hat\psi$ and $1/\hat\psi$ are analytic in $\Omega_{\beta/2}$, the zero at $\xi=0$ of $\F_{x\to\xi}\Psi f=\hat f(\xi)\hat\psi(\xi)$ and of $\F_{x\to\xi}\bar\Psi f=\hat f(\xi)/\hat\psi(\xi)$ is of the same order as of $\hat f$. So $\Psi,\bar\Psi:\E_k\to\E_k$ for all $k\in\N_0$.

By applying the Fourier transform, we see that $\Psi\circ\bar\Psi f=\bar\Psi\circ\Psi f=f$ for all $f\in\E$, i.e.~$\bar\Psi=\Psi^{-1}$, and $\Psi,\Psi^{-1}:\E_k\to\E_k$ are bijections for all $k\in\N_0$.

Finally, as in Corollary \ref{lemma:bdd_theta} one proves the boundedness of $\Psi$ and $\Psi^{-1}$ by using the assumption {\bf (C)(iii)}.
\end{proof}

The map $\Psi$ plays a crucial role in the analysis of the perturbed Fokker-Planck operator $\L+\Theta$, because it relates the eigenspaces of $\L$ to the eigenspaces of $\L+\Theta$: According to Proposition \ref{pert:spec} we have: 
\begin{equation}\label{psi}
f_k=\Psi\mu_k,\quad k\in\N_0.                                                                                                                            
\end{equation}
By using this property of $\Psi$ we obtain the following result:

\begin{prop}\label{aquiv:reso}
Let $k\in\N_0$ and $\zeta\in\C\backslash\{-k,-k-1,\ldots\}$. Then there holds
\begin{equation}\label{equiv:resolvents}
R_{\L+\Theta}(\zeta)|_{\E_k} =\Psi\circ R_\L(\zeta)\circ\Psi^{-1}|_{\E_k}.
\end{equation}
In particular there exists a constant $\tilde C_k>0$ such that
\begin{equation}\label{reso:est}
 \big\|\big(R_{\L+\Theta}(\zeta)|_{\E_k}\big)^n\big\|_{\BB(\E_k)}\le \frac{\tilde C_k}{(\Re \zeta+k)^n},\quad \Re\zeta>{-k},\, n\in\N.
\end{equation}
\end{prop}

\begin{proof}
We fix $k\in\N_0$. Then for all $j\ge k$ and $\zeta\in\C\backslash\{-k,-k-1,\ldots\}$ there holds due to (\ref{psi}):
\[
R_\L(\zeta)\mu_j=\frac{\mu_j}{\zeta+j}=\Psi^{-1}\circ R_{\L+\Theta}(\zeta)f_j= \Psi^{-1}\circ R_{\L+\Theta}(\zeta)\circ\Psi\mu_j.
\]
So we have $ R_\L(\zeta)=\Psi^{-1}\circ R_{\L+\Theta}(\zeta)\circ\Psi$ in the space $\spn\{\mu_j:j\ge k\}\subset E_k$, which is dense in $\E_k$. Then this identity extends to $\E_k$ due to the continuity of the occurring operators.

In order to prove the resolvent estimate (\ref{reso:est}) we use 
\[\big(R_{\L+\Theta}(\zeta)|_{\E_k}\big)^n=R_{\L+\Theta}(\zeta)^n|_{\E_k}=\Psi\circ R_{\L}(\zeta)^n\circ \Psi^{-1}|_{\E_k},\]
which follows from (\ref{equiv:resolvents}) and Lemma \ref{eigs:Psi} (\ref{klam_i}). Because of $\Psi,\Psi^{-1}\in\BB(\E_k)$ we conclude
\begin{equation}\label{circ_estim}
 \big\|\big(R_{\L+\Theta}(\zeta)|_{\E_k}\big)^n\big\|_{\BB(\E_k)}\le \|\Psi\|_{\BB(\E_k)} \big\|\big(R_{\L}(\zeta)|_{\E_k}\big)^n\big\|_{\BB(\E_k)}\|\Psi^{-1}\|_{\BB(\E_k)}.
\end{equation}
Due to the semigroup estimate in Theorem \ref{extension_theorem} (\ref{fp:res:iv}) there holds
\[
 \big\|\big(R_{\L}(\zeta)|_{\E_k}\big)^n\big\|_{\BB(\E_k)}\le \frac{C_k}{(\Re \zeta+k)^n},\quad \Re \zeta>-k,\,n\in\N,
\]
according to the Hille-Yosida theorem. Inserting this estimate in (\ref{circ_estim}) shows (\ref{reso:est}).
\end{proof}

% 
% \begin{rem}\?
%  Since $\Psi,\Psi^{-1}\in\BB(\E)$, the norm $\|\Psi(\cdot)\|_\omega$ is equivalent to $\|\cdot\|_\omega$ on $\E$. Therefore, the map $\Psi:(\E,\|\Psi(\cdot)\|_\omega)\to (\E,\|\cdot\|_\omega)$ is an isometric isomorphism. Thus, according to \eqref{equiv:resolvents} the operator $\L$ in $(\E,\|\Psi(\cdot)\|_\omega)$ is isometrically equivalent to $\L+\Theta$ in $(\E,\|\cdot\|)$.
% \end{rem}

\begin{rem}
 According to \eqref{equiv:resolvents} the operators $\L$ and $\L+\Theta$ are similar:
\[\L+\Theta = \Psi\circ\L\circ\Psi^{-1}.\]
Now we consider the family of operators $(\L(\tau))_{\tau\in\R}:=(\L+\tau\Theta)_{\tau\in\R}$. Clearly, for every $\tau\in\R$ the operators $\L(\tau)$ and $\L(0)=\L$ are similar with the transformation operator $\Psi(\tau)$ defined according to Lemma \ref{eigs:Psi} (where we replace $\theta$ by $\tau\theta$ in \eqref{last_ref}). Therefore, according to \cite{laxx} there exists a family of operators $(B(\tau))_{\tau\in\R}$ such that $(\L(\tau),B(\tau))$ form a {\em Lax pair}, i.e.~they obey
\[\diff{}\tau \L(\tau)=[B(\tau),\L(\tau)],\]
where the right hand side denotes the commutator. Since we explicitly know the transformation operator $\Psi(\tau)$ we can compute $B(\tau)$:
\[Bf:=-\Psi(\tau)\circ\diff{[\Psi(\tau)]^{-1}}{\tau}f=\F^{-1}\Big[\int_0^1\frac{\hat\theta(\xi s)}s\d s\,\hat f\Big],\]
which is independent of $\tau$.
\end{rem}

\begin{cor}
Let $k\in\N_0$. Then there exists a constant $\tilde C_k>0$ such that
\begin{equation}\label{pert_dec_rate}
\big\|\e^{t(\L+\Theta)}|_{\E_k}\big\|_{\BB(\E_k)} \le \tilde C_k\e^{-kt},\quad t\ge 0.
\end{equation}
\end{cor}

\begin{proof}
The result immediately follows from (\ref{reso:est}) by application of the Hille-Yosida theorem.
\end{proof}

\begin{rem}\label{conv_to_f_0}
The above result implies the exponential convergence of any solution of (\ref{pert_fp}) towards the (appropriately scaled) stationary state: Choose any $f\in\E$. Then there exists a unique constant $m\in\C$ (the ``mass'' of $f$) such that $\mathcal P_0 f=m f_0$. So $f-mf_0=(1-\mathcal P_0)f\in\E_1$, cf.~Lemma \ref{def:spec:proj:pert}, which implies $\e^{t(\L+\Theta)}f-mf_0=\e^{t(\L+\Theta)}(f-mf_0)\in\E_1$ for all $t\ge 0$, due to Proposition \ref{prop_3_11}. With (\ref{pert_dec_rate}) and $k=1$ this implies
\[\|\e^{t(\L+\Theta)}f-mf_0\|_\omega\le \tilde C_1\|f-mf_0\|_\omega\e^{-t},\quad t\ge0.\]
\end{rem}

\begin{rem}
In the one dimensional case we can explicitly compute the Fourier transform of $R_{\L+\Theta}(\zeta)g$, see Proposition \ref{f_trafo_reso}: For any $k\in\N_0$, $\Re\zeta>-k$, and $g\in \E_k$, the unique solution $f\in\E_k$ of $(\zeta-\L-\Theta)f=g$ satisfies
\[
 \hat f(\xi)=\F_{x\to\xi} [R_{\L+\Theta}(\zeta)g]=\hat f_0(\xi)\int_0^1\frac{\hat g(s\xi)}{\hat f_0(s\xi)}s^{\zeta-1}\d s,\quad \xi\in\Omega_{\beta/2},
\]
where $s^{\zeta}=\e^{\zeta\log s}$ and $\log$ is the natural logarithm on $\R^+.$ One can use this representation for an alternative proof of the resolvent estimate (\ref{reso:est}). However, this becomes less convenient in higher dimensions, since it is then not clear how to properly compute the explicit Fourier transform of $R_{\L+\Theta}(\zeta)$.
\end{rem}

Now we summarize our results in the final theorem:

\begin{trm}\label{final_pert_trm}
Let $\E=L^2(\omega)$, where $\omega(x)=\cosh\beta x$, for some $\beta>0$, and let $\Theta$ fulfill the condition {\bf (C)} for this $\beta>0$. Then the perturbed operator $\L+\Theta$ has the following properties in $\E$:
\begin{enumerate}
 \renewcommand{\theenumi}{\roman{enumi}}
\renewcommand{\labelenumi}{(\theenumi)}
\item It has compact resolvent, and $\sigma(\L+\Theta)=\sigma_p(\L+\Theta)=-\N_0$.
\item There holds $M(\L+\Theta+k)=\ker(\L+\Theta+k)=\spn\{f_k\}$, where $f_k$ is the eigenfunction to the eigenvalue $-k$ given by
(\ref{rec_f_k}). The eigenfunctions are related by $f_k=f_0^{(k)}$.
\item The spectral projection $\mathcal P_k$ corresponding to the eigenvalue $-k\in-\N$ fulfills
\[ \ran\mathcal P_k=\spn\{f_k\},\quad\ker\mathcal P_k=\E_{k+1}\oplus\spn\{ f_{k-1},\ldots, f_0\},\]
where the $(\L+\Theta)$-invariant spaces $\E_k$ are explicitly given in (\ref{e_k}). Moreover, $\ran\mathcal P_0=\spn\{f_0\}$ and $\ker\mathcal P_0=\E_1$. % In particular, all singularities of the resolvent are poles of order one.
\item For every $k\in\N_0$, the operator $(\L+\Theta)|_{\E_k}$ generates a $C_0$-semigroup in $\E_k$, denoted by $(\e^{t(\L+\Theta)}|_{\E_k})_{t\ge 0}$, which satisfies the estimate
\[\|\e^{t(\L+\Theta)}|_{\E_k}\|_{\BB(\E_k)}\le\tilde  C_k\e^{-k t},\quad t\ge0,\]
where the constant $\tilde C_k>0$ is independent of $t$.
\end{enumerate}
\end{trm}

\begin{rem}
 Apparently, the particular choice of $\beta>0$ has no influence on the above results, except possibly for the constants $\tilde C_k$. In practice, the constant $\beta$ may therefore be chosen arbitrarily small, such that $\Theta$ satisfies {\bf (C)} for this $\beta$.
\end{rem}

\section{The Higher-Dimensional Case}\label{sec35}

As already mentioned in the introduction, the preceding results can be generalized to higher dimensions without much additional effort. Most proofs are analogous to the ones in the one-dimensional case. Therefore we give here only an outline of the steps leading to the extension of Theorem \ref{final_pert_trm} to higher dimensions.

\medskip

In this section we consider the perturbed Fokker-Planck equation (\ref{pert_fp}) on $\R^d$, where $d\in\N$ is the spatial dimension. Elements of $\R^d$ resp.~$\C^d$ are represented by bold letters, e.g.~$\x\in\R^d,\bfxi\in\C^d$, and we write $\x=(x_1,\ldots,x_d)$. The $\ell^1$-norm is $|\x|:=|x_1|+\cdots+|x_d|$. For a multi-index $\kk\in\N_0^d$ we define $\x^\kk:=x_1^{k_1}\cdots x_d^{k_d}$ and $\kk!:=k_1!\cdots k_d!$. Furthermore
 \[D^\kk:=\frac{\pd^{|\kk|}}{\pd x_1^{k_1}\cdots \pd x_d^{k_d}}.\]
We adopt the notation for weighted Sobolev spaces on $\R^d$ from Section \ref{sec2}, as well as the normalization of the Fourier transform.

We consider the Fokker-Planck operator on $\R^d$ given by
\[L f:=\nabla\cdot\bigg(\mu\nabla\bigg(\frac{f}{\mu}\bigg)\bigg)=\Delta f+\x\cdot\nabla f+df,\]
where $\mu(\x):=\exp(-\x\cdot\x/2)$. The natural space to consider $L$ in is $E:=L^2(1/\mu)$. Since it is isometrically equivalent to the harmonic oscillator $H:=-\Delta- d/2+|\x|^2/4$ in $L^2(\R^d)$, we transfer many results of $H$ (see \cite{parmegg} and \cite[Theorem XIII.67]{resi}) to $L$. In the following we summarize some properties of $L$ in $E$ (see also \cite{Metafune2001,bakry,helffernier}):

\begin{trm}\label{prop:fp_in_H:d}
The Fokker-Planck operator $L$ in $E$ has the following properties:
      \begin{enumerate}
      \renewcommand{\theenumi}{\roman{enumi}}
\renewcommand{\labelenumi}{(\theenumi)}
	\item\label{st_fp:1:d} $L$ with $D(L)=\{f\in E:Lf\in E\}$ is self-adjoint and has a compact resolvent.
	\item\label{st_fp:2:d} The spectrum is $\sigma(L)=-\N_0$, and it consists only of eigenvalues.
	\item\label{st_fp:3:d} For each eigenvalue $-k\in\sigma(L)$ the corresponding eigenspace has the dimension $\binom{k+d-1}{k}$, and it is spanned by the eigenfunctions
	      \[
		\mu_\kk(\x):=\prod_{\ell=1}^d \mu_{k_\ell}(x_\ell),\quad |\kk|=k,
	      \]
	where the $\mu_j$ are defined in Theorem \ref{prop:fp_in_H}.
	\item\label{st_fp:4:d} The eigenfunctions $(\mu_\kk)_{\kk\in\N^d_0}$ form an orthogonal basis of $E$.
	\item\label{st_fp:45:d} The spectral projection $\Pi_{L,k}$ onto the $k$-th eigenspace is given by
	      \[
	       \Pi_{L,k}=\sum_{|\kk|=k}\Pi_{L,\kk},\quad\text{where}\quad \Pi_{L,\kk}:=\frac{(2\pi)^{d/2}}{\kk!}\mu_\kk\la\cdot,\mu_\kk\ra_{E}.
	      \]
	      There holds the spectral representation $L=\sum_{k\in\N_0}-k\Pi_{L,k}.$
	\item\label{st_fp:5:d} The operator $L$ generates a $C_0$-semigroup of contractions on $E_k$ for all $k\in\N_0$, where $E_k:=\ker(\Pi_{L,0}+\cdots+\Pi_{L,k-1}),\,k\ge 1$, and $E_0:=E$. The semigroup satisfies the estimate
	      \[\big\|\e^{tL}|_{E_k}\big\|_{\BB(E_k)}\le\e^{-kt},\quad\forall k\in\N_0.\]
      \end{enumerate}
\end{trm}

The next step is to properly define $L$ in $\E:=L^2(\omega)$ with a weight
 \[\omega(\x)=\sum_{j=1}^d \cosh\beta x_j,\]
  with $\beta>0$. As in the one-dimensional case we have a characterization of $\E$ by the Fourier transform. Due to (a small variant of) \cite[Theorem IX.13]{resi2} we have: There holds $f\in\E$ iff $\hat f$ has an analytic continuation (denoted by $\hat f$ as well) to the set $\Omega_{\beta/2}:=\{\mathbf z\in\C^d:|\Im \mathbf z|<\beta/2\}$ and
\begin{equation}\label{f_in_E:d}
 \sup_{\substack{|\mathbf b|<\beta/2\\ \mathbf b\in\R^d}}\|\hat f(\cdot+\ii\mathbf b)\|_{L^2(\R^d)}<\infty.
\end{equation}

For any $\mathbf b\in\R^d$ with $|\mathbf b|<\beta /2$ we have $\hat f(\bfxi+\ii\mathbf b)=\F_{\x\to\bfxi}\big(\e^{\mathbf b\cdot\x}f(\x)\big).$
The right hand side still makes sense for $|\mathbf b|=\beta/2$ as an $L^2(\R^d)$-function. And according to this identity and Plancherel's formula there holds $\mathbf b\mapsto \hat f(\cdot+\ii\mathbf b)\in C(\overline{B(\beta/2,0)};L^2(\R^d))$, where $B(\beta/2,0):=\{\mathbf b\in\R^d:|\mathbf b|<\beta/2\}$. We can use this fact to define the norm
\begin{equation}\label{f_norm:d}
 \nn f\nn^2_\omega:=\sum_{\ell=1}^d \Big\|\hat f\Big(\bfxi+\ii\frac\beta2\mathbf{e}_\ell\Big)\Big\|^2_{L^2(\R^d_{\bfxi})}+\Big\|\hat f\Big(\bfxi-\ii\frac\beta2\mathbf{e}_\ell\Big)\Big\|^2_{L^2(\R^d_{\bfxi})},
\end{equation}
where $\mathbf{e}_\ell\in\R^d$ is the $\ell$-th unit vector in $\R^d$. The norm $\nn\cdot\nn_\omega$ is equivalent to $\|\cdot\|_\omega$.

In $\E$ there holds a Poincar\'e inequality:

\begin{lem}\label{pr_pc:d}
Tthere exists a constant $C>0$ such that for all $f\in C_0^\infty(\R^d)$:
\begin{equation}\label{poincare:d}
 \|f\|_\omega\le C\|\nabla f\|_\omega.
\end{equation}
\end{lem}
For the proof see Appendix \ref{app:proof}. A similar statement is given in \cite[Theorem 14.5]{GOIII}. By using this Poincar\'e inequality we can generalize Lemma \ref{reso_estim}: Let again $\mathfrak L =\Delta+\x\cdot\nabla+d$ be the distributional Fokker-Planck operator. For $f,g\in\E\subset\sw'(\R^d)$ with $(\zeta-\mathfrak L)f=g$ we have the estimate
\begin{equation}\label{djan}
 \|f\|_\varpi+\|\nabla f\|_\omega\le C\|g\|_\omega,
\end{equation}
where $\varpi(\x)=(2\Re\zeta-d)\omega(\x)+\x\cdot\nabla\omega(\x)-\Delta\omega(\x)$, which is a weight function for $\Re\zeta$ sufficiently large. Now we may proceed analogously to the proof of Lemma \ref{abschluss_LL} and show that $\mathfrak L|_{C_0^\infty(\R^d)}$ is closable in $\E$, and its closure $\L$ has the domain $D(\L)=\{f\in\E:\mathfrak L f\in\E\}$. From \cite[Theorem 2.4]{Opic1989} we get the compact embedding $W^{1,2}(\varpi,\omega)\inj\inj\E$, and together with the estimate \eqref{djan} this implies the compactness of the resolvent of $\L$, analogously to Theorem \ref{trm:compactness}. Hence, the spectrum of $\L$ consists only of eigenvalues, and there holds:

\begin{lem}
 In $\E$ we have $\sigma(\L)=-\N_0$. The eigenspaces are still spanned by the $\mu_\mathbf k$.
\end{lem}

\begin{proof}[Proof (Sketch).]
We consider the Fourier transform of the eigenvalue equation $\mathfrak L f=\zeta f$, and by setting $\tilde f(\bfxi):=\hat f(\bfxi)/\hat \mu(\bfxi)$ we get analogously to the calculation in the Appendix \ref{app:b} the equation
\begin{equation}\label{simpl_eveq}
 \bfxi\cdot\nabla\tilde f(\bfxi)=-\zeta\tilde f(\bfxi).
\end{equation}
For each $j\in\{1,\ldots,d\}$ the function $\tilde f(0,\ldots,0,\xi_j,0,\ldots, 0)$ needs to be analytic in $\Omega_{\beta/2}$, and satisfies \eqref{ode:tilde_f} for $\tilde g=0$. So, as in the Appendix \ref{app:b} we find that it is necessary that $\zeta\in -\N_0$.

For $k:=-\zeta\in\N_0$ and $\bfxi\in\R^d$ we obtain by differentiating \eqref{simpl_eveq} with respect to $\xi_j$:
\[
 \bfxi\cdot\nabla\Big(\pdiff{\tilde f(\bfxi)}{\xi_j}\Big)=(k-1)\Big(\pdiff{\tilde f(\bfxi)}{\xi_j}\Big).
\]
Thus, for any $\kk\in\N_0^d$ with $|\kk|=k$ we get
\[
 \bfxi\cdot\nabla\big(D^{\kk}\tilde f(\bfxi)\big)=0,
\]
and all characteristics meet at $\bfxi=\mathbf 0$. 
$\hat f$ is analytic on $\R^d$. Hence, the continuity of $D^\kk\tilde f(\bfxi)$ at $\bfxi=\mathbf 0$ implies $D^{\kk}\tilde f(\bfxi)=C$ for some constant $C\in\C$. This holds for any $|\kk|=k$, so the general solution of \eqref{simpl_eveq} is a linear combination of all $\bfxi^{\kk}$ with $|\kk|=-\zeta=k$. Therefore, the Fourier transform of an eigenfunction $f$ with $(\L+k) f=0$ is a linear combination of the $\bfxi^{\kk}\mu(\bfxi)$ with $|\kk|=k$ (and, equivalently, $f(\x)$ is a linear combination of the $D^{\kk}\mu(\x)$). Then, according to Theorem \ref{prop:fp_in_H:d} \eqref{st_fp:3:d} and Theorem \ref{prop:fp_in_H} \eqref{st_fp:3}, the eigenspace for $\zeta=-k$ is spanned by the $\mu_\kk$.
\end{proof}

As in Proposition \ref{spec_pr} we can define the $\L$-invariant subspaces $\E_k:=\cl_\E E_k=\cl_\E\spn\{\mu_\kk:|\kk|\ge k\}$ for all $k\in\N_0$, and $\sigma(\L|_{\E_k})=\{-k,-k-1,\ldots\}$. By applying Lemma \ref{lem:funct} we get by induction
\begin{align}
  \E_k&=\Big\{f\in\E:\int_{\R^d} f(\x)\x^\kk\d\x=0, \, |\kk|\le k-1\Big\}\label{e_k:d_a}\\
&=\big\{f\in\E:D^\kk\hat f(0)=0, \, |\kk|\le k-1\big\}\nonumber.
\end{align}
Analogously to Proposition \ref{spec_prop} \eqref{ho:ii} we can also characterize the spectral projections corresponding to the eigenvalues $-k\in-\N_0$, see the result of Theorem \ref{extension_theorem:d} \eqref{fp:res:iii:d} below. Finally, as in the one-dimensional case, one shows that $\L$ generates a $C_0$-semigroup of bounded operators $(\e^{t\L})_{t\ge 0}$, which is given by the formula (cf.~\cite[Appendix A]{Gallay2002})
\[
 \F_{\x\to\bfxi}\big[\e^{t\L}f\big]=\exp\Big(-\frac{\bfxi\cdot\bfxi}2(1-\e^{-2t})\Big)\hat f\big(\bfxi\e^{-t}\big),\quad t\ge0.
\]
The corresponding decay estimates on the subspaces $\E_k$ can be shown similar to the proof of Proposition \ref{unpert_decay}. For this one uses the norm \eqref{f_norm:d} and the  behavior of $\hat f$ around the origin $\bfxi=\mathbf 0$ for $f\in\E_k$. For a rigorous proof see also \cite{AchArSt}.

\begin{trm}\label{extension_theorem:d}
In $\E:=L^2(\omega)$, with $\omega(\x)=\cosh\beta|\x|$ and $\beta>0$, the operator $L$ is closable, and $\L:=\cl_\E L$ has the following properties:
\begin{enumerate}
\renewcommand{\theenumi}{\roman{enumi}}
\renewcommand{\labelenumi}{(\theenumi)}
 \item\label{fp:res:i:d} The spectrum satisfies $\sigma(\L)=-\N_0$, and $M(\L+k)=\ker(\L+k)=\spn\{\mu_\kk:|\kk|=k\}$ for any $k\in\N_0$. The eigenfunctions satisfy $\mu_\kk=D^\kk\mu_{\mathbf 0}$.
 \item\label{fp:res:ii:d} For any $k\in\N_0$ the closed subspace $\E_k:=\cl_\E\spn\{\mu_\kk:|\kk|\ge k\}$ is an $\L$-invariant subspace of $\E$, and $\spn\{\mu_\kk:|\kk|\le k-1\}$ is a complement. In particular $\E_0=\E$.
\item\label{fp:res:iii:d} The spectral projection $\Pi_{\L,k}$ to the eigenvalue $-k\in-\N_0$ fulfills $\ran\Pi_{\L,k}=\spn\{\mu_\kk:|\kk|=k\}$ and $\ker\Pi_{\L,k}=\E_{k+1}\oplus \spn\{\mu_\kk:|\kk|\le k-1\}$.
 \item\label{fp:res:iv:d} For any $k\in\N_0$ the operator $\L$ generates a $C_0$-semigroup on $\E_k$, and there exists a
constant $C_k\ge 1$ such that we have the estimate
\[
 \left\|\e^{t\L}|_{\E_k}\right\|_{\BB(\E_k)}\le C_k \e^{-kt},\quad\forall t\ge 0.
\]
\end{enumerate}
\end{trm}

Next we specify the conditions on the perturbation $\Theta$.

\medskip
\noindent{\bf ($\mathbf{C_d}$) Conditions on $\boldsymbol\Theta$:} We assume that $\Theta f=\theta*f$, for $f\in\E$, where $\theta$ is a tempered distribution that fulfills the following properties in $\Omega_{\beta/2}$ for some $\beta>0$:
\begin{enumerate}
\renewcommand{\theenumi}{\roman{enumi}}
\renewcommand{\labelenumi}{(\theenumi)}
	\item The Fourier transform $\hat \theta$ can be extended to an analytic function in $\Omega_{\beta/2}$ (also denoted by $\hat \theta$), and $\hat\theta\in L^\infty(\Omega_{\beta/2})$.
	\item It holds $\hat\theta(\mathbf 0)=0$, i.e.~$\theta$ has zero mean.
	\item The mapping $\bfxi\mapsto\Re\int_0^1\hat\theta(\bfxi s)/s\d s$ is essentially bounded in $\Omega_{\beta/2}$.
\end{enumerate}

Condition {\bf ($\mathbf{C_d}$)(i)} ensures that $\Theta\in\BB(\E)$, which is seen by using the norm $\nn\cdot\nn_\omega$. And due to {\bf ($\mathbf{C_d}$)(ii)} we have $\Theta:\E_k\to\E_{k+1}$ for all $k\in\N_0$. In the following we always assume that {\bf ($\mathbf{C_d}$)} holds.

\begin{prop}\label{pert:spec:d}
We have the following spectral properties of $\L+\Theta$ in $\E$:
\begin{enumerate}\renewcommand{\theenumi}{\roman{enumi}}
\renewcommand{\labelenumi}{(\theenumi)}
 \item\label{eins} $\sigma(\L+\Theta)=-\N_0$.
 \item\label{zwei} For each $k\in\N_0$, the eigenspace $\ker(\L+\Theta+k)$ has the dimension $\binom{k+d-1}{k}$.
 \item Under appropriate scaling, the eigenfunctions $f_\kk$ to the eigenvalue $-k\in\-\N_0$ are explicitly given by
\begin{equation}\label{rec_f_k:d}
 f_\kk=D^\kk f_{\mathbf 0},\quad|\kk|=k,
\end{equation}
where
\begin{equation}\label{f_k:d}
 \hat f_{\mathbf 0}(\bfxi):=\exp\Big(-\frac{\bfxi\cdot\bfxi}2+\int_0^1\frac{\hat\theta(\bfxi s)}s\d s\Big),\quad \bfxi\in\Omega_{\beta/2}\subset\C^d.
\end{equation}
\end{enumerate}
Thereby $f_{\mathbf 0}$ is the unique stationary solution of the perturbed Fokker-Planck equation (\ref{pert_fp}) with unit mass.
\end{prop}

\begin{proof}[Proof (Sketch)]
Since the resolvent is compact (see the discussion above), the spectrum consists only of eigenvalues. As in the one-dimensional case one shows $\sigma(\L+\Theta)\subseteq -\N_0$ by applying $\Pi_{\L,k}$ to the eigenvalue equation. This also implies $\dim\ker(k+\L+\Theta)\le\dim\ran\Pi_{\L,k}=\binom{k+d-1}{k}$. Then one verifies that the functions $f_\kk$ given in (\ref{rec_f_k:d}) are eigenfunctions, and lie in $\E$, according to the condition (\ref{f_in_E:d}). Since $\dim\spn\{f_\kk:|\kk|=k\}=\binom{k+d-1}{k}$, there are no further eigenfunctions, due to the previous estimate on the dimension of the eigenspaces. So $\ker(k+\L+\Theta)=\spn\{f_\kk:|\kk|=k\}$ for all $k\in\N_0$.
\end{proof}

Now we introduce 
\[\hat \psi(\bfxi):=\exp\Big(\int_0^1\frac{\hat\theta(\bfxi s)}s\d s\Big),\quad \bfxi\in\Omega_{\beta/2},\]
and the mapping $\Psi:f\mapsto f*\psi$. The results of Lemma \ref{eigs:Psi} for $\Psi$ still hold, and due to (\ref{f_k:d}) we have for all $\kk\in\N_0$:
\[f_\kk=\Psi\mu_\kk.\]
As in Proposition \ref{aquiv:reso} we obtain $R_{\L+\Theta}(\zeta)|_{\E_k} =\Psi\circ R_\L(\zeta)\circ\Psi^{-1}|_{\E_k},$ for all $k\in\N_0$ and $\zeta\in\C\backslash\{-k,-k-1,\ldots\}.$ The estimates (\ref{reso:est}) and (\ref{pert_dec_rate}) also hold here, and for the convergence of $f(t)=\e^{t(\L+\Theta)}f$ to the stationary solution see Remark \ref{conv_to_f_0}. As in Section \ref{sec3} we finally have:

\begin{trm}\label{final_pert_trm:d}
Let $\E=L^2(\omega(\x)\d \x)$, where $\omega(\x)=\cosh\beta |\x|$, for some $\beta>0$ and $\x\in\R^d$, and let $\Theta$ fulfill the condition {\bf ($\mathbf{C_d}$)} for this $\beta>0$. Then the perturbed operator $\L+\Theta$ has the following properties in $\E$:
\begin{enumerate}
 \renewcommand{\theenumi}{\roman{enumi}}
\renewcommand{\labelenumi}{(\theenumi)}
\item It has compact resolvent, and $\sigma(\L+\Theta)=\sigma_p(\L+\Theta)=-\N_0$.
\item There holds $M(\L+\Theta+k)=\ker(\L+\Theta+k)=\spn\{f_\kk:|\kk|=k\}$, where the $f_\kk$ are the eigenfunctions given by
(\ref{rec_f_k:d}). They are related by $f_\kk=D^\kk f_{\mathbf 0}$.
\item The spectral projection $\P_{k}$ to the eigenvalue $-k\in-\N_0$ fulfills $\ran\P_{k}=\spn\{f_\kk:|\kk|=k\}$ and $\ker\P_{k}=\E_{k+1}\oplus \spn\{f_\kk:|\kk|\le k-1 \}$, where the $(\L+\Theta)$-invariant spaces $\E_k$ are explicitly given in (\ref{e_k:d_a}).
\item For every $k\in\N_0$, the operator $(\L+\Theta)|_{\E_k}$ generates a $C_0$-semigroup in $\E_k$, denoted by $(\e^{t(\L+\Theta)}|_{\E_k})_{t\ge 0}$, which satisfies the estimate
\[\|\e^{t(\L+\Theta)}|_{\E_k}\|_{\BB(\E_k)}\le\tilde  C_k\e^{-k t},\quad t\ge0,\]
where the constant $\tilde C_k>0$ is independent of $t$.
\end{enumerate}
\end{trm}

\section{Simulation Results}\label{sec4}

In this section we shall illustrate numerically the exponential convergence for the one-di\-men\-sio\-nal perturbed Fokker-Planck equation (\ref{pert_fp}), with $\theta:=\epsilon(\delta_{-\alpha}-\delta_{\alpha})$, i.e.~$\Theta f(x)=\epsilon(f(x+\alpha)-f(x-\alpha))$, for some $\epsilon,\alpha\in\R$. 
The eigenfunctions $f_k$ of the evolution operator $\L+\Theta$ can be obtained by an inverse Fourier transform, with $\hat f_k$ explicitly given in (\ref{rec_f_k}). If the initial condition $\phi$ is a (finite) linear combination of the $f_k$, the solution to (\ref{pert_fp}) reads explicitly 
\[
f(t,x)=\e^{t(\L+\Theta)}\Big[\sum_{j=1}^n a_jf_{k_j}\Big]=\sum_{j=1}^n a_j\e^{-k_j t}f_{k_j},\quad \forall t\ge 0.\]

In the simulation we use a mass conserving Crank-Nicolson finite difference scheme for (\ref{pert_fp}). It is employed on the spatial interval $[-25,25]$ (with $1500$ gridpoints) along with
zero-flux boundary conditions. Moreover, we choose $\alpha=\epsilon=2$ and $\beta=1$, i.e.\ $\E=L^2(\cosh x)$.

\begin{figure}[h!]
 \subfigure[Initial condition $\phi_1=(f_1-1.32f_2)/\|f_1-1.32f_2\|_\omega$.]{\label{fig_a}\includegraphics[width=0.49\textwidth]{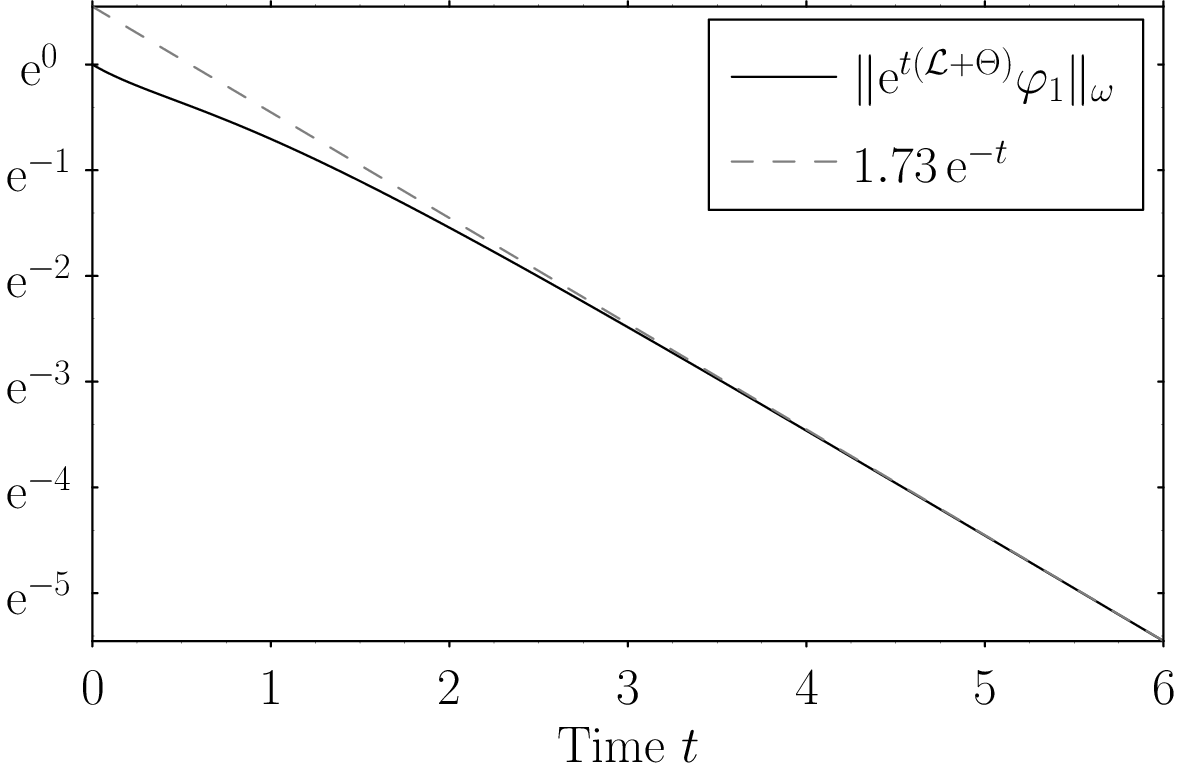}}\hfill
 \subfigure[Initial condition $\phi_2=(\chi_{[-4,0]}-\chi_{[0,4]})/\allowbreak\|\chi_{[-4,0]}-\chi_{[0,4]}\|_\omega\in\E_1$.]{\label{fig_b}\includegraphics[width=0.49\textwidth]{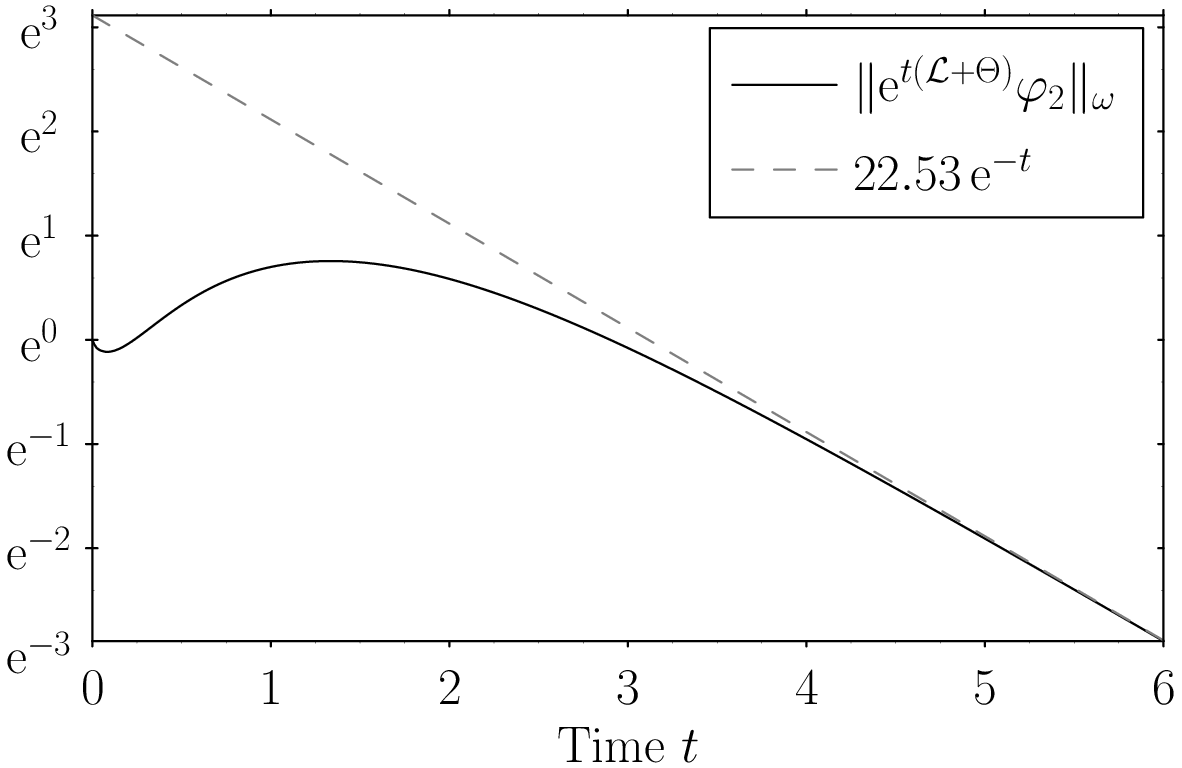}}\caption{Evolution of the norm $\|\cdot\|_\omega$ of solutions of the perturbed equation for different initial conditions $\phi$.}
\end{figure}

The following numerical results verify the decaying behaviour of solutions to (\ref{pert_fp}), and yield an estimate to the constants $\tilde C_k$ from Theorem \ref{final_pert_trm}. 
First we consider the initial condition $\phi_1=(f_1-1.32f_2)/\|f_1-1.32f_2\|_\omega$. 
For the corresponding solution we plot $\|f(t,\cdot)\|_\omega$ in Figure \ref{fig_a}.
Since the sequence $(f_k)_{k\in\N}$ is not orthogonal in $\E$, the initial decay rate is here smaller than the individual decay rate of $f_1$ (i.e.~$-1$). But after some time, the $f_1$-term becomes  dominant, and the decay rate approaches $-1$. For large times, the norm behaves approximately like $1.73 \,\e^{-t},$ so we have the lower bound $\tilde C_1\ge 1.73$.

As a second example we choose the initial condition $\phi_2=(\chi_{[-4,0]}-\chi_{[0,4]})/\allowbreak\|\chi_{[-4,0]}-\chi_{[0,4]}\|_\omega$. It lies in $\E_1$ since it is massless. The evolution of $\|f(t,\cdot)\|_\omega$ is displayed in Figure \ref{fig_b}. Here, the norm even increases initially. Only after some time, the norm begins to decay with a rate tending to $-1$. For large times $t$, the norm behaves approximately like $22.53\,\e^{-t}$, which shows  $\tilde C_1\ge 22.53$.
\appendix

\section{Spectral Projections}\label{app:space_enlarg}

In this section we review some properties of spectral projections and resolvents, cf.~\cite[Chapters V.9-10]{taylay}, \cite[Chapter VIII.8]{yosida} and \cite[Sections III.6.4-5]{kato}.\medskip

Here, $X$ is a Hilbert space, $A\in\CC(X)$, and we assume $\lambda\in\sigma(A)$ to be an isolated point 
of the spectrum. Then the corresponding spectral projection $\Rho_{A,\lambda}$ is defined by (\ref{def:spec_proj}), and $\lambda$ is an
isolated singularity of the resolvent $R_A(\zeta)$.

\begin{prop}\label{fds}
For every $n\in\N$ we have
\begin{align*}
 \ran(\lambda-A)^n\supseteq & \ker \Rho_{A,\lambda},\\
 \ker(\lambda-A)^n\subseteq & \ran \Rho_{A,\lambda}.
\end{align*}
There exists some $n\in\N$ such that both inclusion relations become equalities iff $\lambda$ is a pole of $R_A(\zeta)$. In
this case $\lambda\in\sigma_p(A)$, i.e.~an eigenvalue. 
\end{prop}

\begin{prop}\label{spec_prop}
For the reduction of $A$ by a fixed spectral projection $\Rho_{A,\lambda}$ we have:
\begin{enumerate}
\renewcommand{\theenumi}{\roman{enumi}}
\renewcommand{\labelenumi}{(\theenumi)}
\item There holds $\Rho_{A,\lambda}D(A)\subset D(A)$, and $\ker \Rho_{A,\lambda}$ and $\ran \Rho_{A,\lambda}$ are $A$-invariant subspaces of $X$.
\item $A|_{\ran \Rho_{A,\lambda}}\in\CC({\ran \Rho_{A,\lambda}})$ and $A|_{\ker \Rho_{A,\lambda}}\in\CC({\ker \Rho_{A,\lambda}})$.
\item There holds $\sigma(A|_{\ran \Rho_{A,\lambda}})=\{\lambda\}$ and $\sigma(A|_{\ker \Rho_{A,\lambda}})=\sigma(A)\backslash\{\lambda\}$. Furthermore $A|_{\ran \Rho_{A,\lambda}}\in\BB(\ran \Rho_{A,\lambda})$.

\item\label{itemiii} If $\dim \ran \Rho_{A,\lambda}<\infty$, then $\lambda-A|_{\ran \Rho_{A,\lambda}}$ is nilpotent, $\lambda$ is a pole of $R_A(\zeta)$,
and $\lambda\in\sigma_p(A)$.
\item\label{itemiv} If $\lambda$ is a pole of $R_A(\zeta)$, then $M(\lambda-A)=\ker(\lambda-A)$ iff the pole has order one.
\end{enumerate}
\end{prop}

For a finite number of isolated points of the spectrum we have:

\begin{lem}\label{uniqueness}
For $N\in\N_0$, let $A$ have isolated points of the spectrum $\zeta_0,\ldots, \zeta_{N-1}$, which are eigenvalues with $\dim
M(\zeta_k-A)<\infty$ for all $0\le k\le N-1$. Assume there exists a closed subspace $Y\subset X$,
such that
 \begin{enumerate}
\renewcommand{\theenumi}{\roman{enumi}}
\renewcommand{\labelenumi}{(\theenumi)}
    \item\label{1:i} $Y$ is $A$-invariant, and $\sigma(A|_{Y})\cap \{\zeta_0,\ldots,\zeta_{N-1}\}=\emptyset$.
    \item\label{1:ii} $X$ can be decomposed as $X=Y\oplus M(\zeta_0-A)\oplus\ldots\oplus M(\zeta_{N-1}-A)$.
 \end{enumerate}
Then $Y=\ker \Pi_A$, where $\Pi_A:=\Pi_{A,0}+\cdots+\Pi_{A,N-1}$ is the sum of the spectral projections $\Pi_{A,k}$ corresponding
to the $\zeta_k$, and $M(\zeta_k-A)=\ran\Pi_{A,k}$ for all $0\le k\le N-1$.
\end{lem}

\begin{proof}
According to the assumptions there holds $\sigma(A|_Y)=\sigma(A)\backslash\{\zeta_0,\ldots,\zeta_{N-1}\}$, and therefore the map $\zeta\mapsto R_A(\zeta)|_{Y}$ is analytic in $\rho(A)\cup\{\zeta_0,\ldots,\zeta_{N-1}\}$. Due to the definition (\ref{def:spec_proj}) of spectral projections this implies that $\Pi_{A,k}Y\equiv0$ for every $\Pi_{A,k}$, and therefore $Y\subseteq \ker \Pi_A$. On the other hand we have $M(\zeta_k-A)\subseteq\ran\Pi_{A,k}$ for all $0\le k\le N-1$, according to Proposition \ref{fds}. From (\ref{1:ii}) we conclude that the inclusions have to be equalities, otherwise $\ker\Pi_A\cap\ran\Pi_A\neq\{0\}$, which is impossible.
\end{proof}

\section{Fourier Transform of the Resolvent}\label{app:b}

This section deals with the explicit computation of the Fourier transform of the resolvent $R_{\L+\Theta}(\zeta)$ of the (one-dimensional) perturbed Fokker-Planck operator $\L+\Theta$ in $\E$, where $\Theta$ fulfills the condition {\bf (C)}. We begin by considering the resolvent equation
\[
 (\zeta-\L-\Theta)f=g
\]
on $\R$, where we assume $\Re\zeta>{-k}$ and $f,g\in\E_k$ for some $k\in\N_0$. We apply the Fourier transform, which yields the following differential equation:
\[
 \xi\Big[\hat f'(\xi)+\Big(\xi+\frac{\zeta-\hat\theta(\xi)}{\xi}\Big)\hat f(\xi)\Big]=\hat g(\xi).
\]
By defining $\tilde f:=\hat f/\hat f_0$ and $\tilde g:=\hat g/\hat f_0$ we obtain the equivalent equation
\begin{equation}\label{ode:tilde_f}
 \xi\tilde f'(\xi)+\zeta\tilde f(\xi)=\tilde g(\xi).
\end{equation}
The general solution for $\xi\in\R^\pm$ reads 
\begin{equation}\label{tf}
 \tilde f(\xi)=\int_0^1\tilde g(\xi s)s^{\zeta-1}\d s+C_\pm\xi^{-\zeta}=:I(\xi)+C_\pm\xi^{-\zeta},
\end{equation}
where the $C_\pm\in\C$ are integration constants to be determined.

First we shall show that the integral $I(\xi)$ is an analytic function on $\Omega_{\beta/2}$: If $g\in\E_k$, then $\tilde g$ is analytic in $\Omega_{\beta/2}$ and has a zero at $\xi=0$ of order not less than $k$, see (\ref{f_char_e_k}). Therefore, for any fixed $\zeta\in\C$ with $\Re\zeta>-k$,
\[
 \tilde g(\xi s)s^{\zeta-1}=\frac{\tilde g(\xi s)}{s^k}s^{\zeta+k-1},\quad s \in (0,1],
\]
is locally integrable at $s=0$, and $I(\xi)$ is well defined for all $\xi\in\Omega_{\beta/2}$. To see that it is actually analytic, we define
$I_\epsilon(\xi):=\int_\epsilon^1 G_k(\xi,s)s^{\zeta+k-1}\d s$ for $\epsilon\in[0,1)$, where
\[G_k(\xi, s):=\begin{cases}
                    \displaystyle \frac{\tilde g(\xi s)}{ s^k},\quad &s\in(0,1],\\
		\displaystyle\frac{\tilde g^{(k)}(0)\xi^k}{k!},\quad& s=0,
                   \end{cases}
\]
for $\xi\in\Omega_{\beta/2}$. The function $G_k(\cdot,s)$ is analytic in $\Omega_{\beta/2}$ for all (fixed) $s\in[0,1]$, and  $G_k$ is continuous in $\Omega_{\beta/2}\times [0,1]$. According to \cite[Theorem 4.9.1]{dettman}, the functions $I_\epsilon(\xi)$ are analytic in $\Omega_{\beta/2}$ for all $\epsilon\in(0,1)$. Now we show that $(I_\epsilon)_{\epsilon\in(0,1)}$ converges normally to $I$ in $\Omega_{\beta/2}$ as $\epsilon\to 0$: Let $K\subset\Omega_{\beta/2}$ be compact. Then we have
\begin{align}
\sup_{\substack{\xi\in K\\s\in[0,1]}}|G_k(\xi,s)|&\le\sup_{\substack{\xi\in K_0\\s\in[0,1]}}|G_k(\xi,s)|=\sup_{\substack{\xi\in K_0\backslash\{0\}\\s\in(0,1]}}\Big|\frac{\tilde g(\xi s)}{(\xi s)^k}\xi^k\Big|\nonumber\\
& \le \sup_{\xi\in K_0\backslash\{0\}}\Big|\frac{\tilde g(\xi )}{\xi ^k}\Big|\cdot\sup_{\xi\in K_0}|\xi^k|=:C_K<\infty,\label{ce_ka}
\end{align}
since $\tilde g(\xi)/\xi^k$ is analytic in $\Omega_{\beta/2}$ (the singularity at $\xi=0$ is removable). Thereby, $ K_0$ is an appropriate convex, compact set with $\{0\}\cup K\subseteq K_0\subset \Omega_{\beta/2}$, and $C_K>0$ is a constant. With (\ref{ce_ka}) we obtain the following estimate for $\xi\in K$ and $0<\epsilon\le 1$:
\[|I(\xi)-I_\epsilon(\xi)|=\Big|\int_0^\epsilon G_k(\xi,s)s^{\zeta+k-1}\d s\Big|\le C_K \frac{\epsilon^{\Re\zeta+k}}{\Re\zeta+k}.\]
Since $\Re\zeta+k>0$, this shows the normal convergence of the analytic functions $I_\epsilon$ towards $I$. According to \cite[Theorem 4.2.3]{dettman} this implies that $I(\xi)$ is analytic in $\Omega_{\beta/2}$.

Now it remains to determine the constants $C_\pm$ in (\ref{tf}). If we require $f\in\E_k$, it is necessary that $\tilde f$ is analytic in $\Omega_{\beta/2}$ and has a zero of order not less than $k$ at $\xi=0$. As already shown, $I(\xi)$ is analytic in $\Omega_{\beta/2}$. Furthermore, for $g\in\E_k$ and all (fixed) $s\in[0,1]$, $\xi\mapsto G_k(\xi,s)$ has a zero of order not less than $k$ at $\xi=0$. Therefore $I(\xi)=\int_0^1G_k(\xi,s)s^{\zeta+k-1}\d s$ has the same property, so $\F^{-1}I\in\E_k$. Thus, it is sufficient to consider the term $C_\pm\xi^{-\zeta}$. If $\zeta\notin-\N_0$, then $\xi^{-\zeta}$ is not analytic in $\Omega_{\beta/2}$ anyway, hence $C_+=C_-=0$. If $\zeta\in\{-k+1,\ldots,-1\}$ for $g\in\E_k$, $\xi^{-\zeta}$ is analytic, and we obtain $C_+=C_-$ because we require continuity of the solution. But the order of the zero of $\xi^{-\zeta}$ is at most $k-1$. Since we need a zero of at least order $k$, we again obtain $C_+=C_-=0$. The conclusion of the above analysis is summarized in the 
following proposition:

\begin{prop}\label{f_trafo_reso}
Let $g\in\E_k$ for some $k\in\N_0$, and $\Re\zeta>{-k}$. Then the unique $f\in\E_k$ with $f=R_{\L+\Theta}(\zeta)g$ satisfies
\[
 \hat f(\xi)=\hat f_0(\xi)\int_0^1\frac{\hat g(s\xi)}{\hat f_0(s\xi)}s^{\zeta-1}\d s,\quad\xi\in\Omega_{\beta/2}.
\]
\end{prop}

\section{Deferred Proofs and Lemmata}\label{app:proof}

\begin{proof}[\underline{Proof of Lemma \ref{analyticity}}]
For $f\in\E$ there holds $f(x)\e^{bx}\in L^2(\R)$ for all $b\in[-\frac\beta 2,\frac\beta2]$. Therefore $\hat f$ is analytic in $\Omega_{\beta/2}$ according to \cite[Theorem IX.13]{resi2}. Due to part (b) of the proof of that theorem (see page~132 in \cite{resi2}), Result (\ref{analyt:ii}) follows. We proceed to the proof of (\ref{analyt:i}). If $f\in \E$ and $b\in [-\frac\beta 2,\frac\beta2]$, we clearly have
\[\|f(x)\e^{bx}\|_{L^2(\R)} \le \|f(x)\e^{\frac\beta 2|x|}\|_{L^2(\R)} \le \sqrt 2\|f\|_\omega.\]
On the left hand side we insert the identity from \eqref{analyt:ii} and use Plancherel's identity, which shows \eqref{seid}. Conversely, let us now assume that $\hat f$ is analytic in $\Omega_{\beta/2}$ and that \eqref{seid} holds. We shall now show that $f\in\E$.  Due to these assumptions we conclude from \cite[Theorem IX.13]{resi2} that $f(x)\e^{bx}\in\ L^2(\R)$ for all $b\in (-\frac\beta 2,\frac\beta2)$. For these values of $b$ we may therefore use the representation from (\ref{analyt:ii}). We insert it in \eqref{seid} and after applying Plancherel's identity we get
\begin{equation}\label{sup_f_hat_e}
 \sup_{\substack{|b|<\beta/2\\b\in\R}}\|f(x)\e^{b|x|}\|_{L^2(\R)}<\infty.
\end{equation}
But this is only possible if $f\in\E$, otherwise the supremum in \eqref{sup_f_hat_e} would not be finite.

Finally we show (\ref{analyt:iii}). For $f\in\E$ there holds $f(x)\e^{\pm\frac \beta 2x}\in L^2(\R)$, and therefore $\xi\mapsto\hat f(\xi\pm\ii\beta/2)$, as defined in (\ref{hatf}), is again an element of $L^2(\R)$. With this definition we now show $b\mapsto\hat f(\cdot+\ii b)\in C([-\beta/ 2,\beta /2];L^2(\R))$. Due to Plancherel's identity we may show equivalently that  $b\mapsto f(x)\e^{bx}$ is continuous in $L^2(\R)$. To this end we fix $b,b_0\in[-\beta/2,\beta/2]$, and we split the integral for any $R>0$:
\begin{align}
 \|f(x)\e^{bx}-f(x)\e^{b_0x}\|_{L^2(\R)}^2 &= \int\displaylimits_{\R\backslash[-R,R]}\!\!|f(x)|^2(\e^{b_0x}-\e^{bx})^2\d x\label{1_2}\\
&\phantom{=}+\int_{-R}^R|f(x)|^2(\e^{b_0x}-\e^{bx})^2\d x\nonumber
\end{align}
Now, for any $\epsilon>0$ we can find some $R=R(\epsilon)>0$ so that $\int_{\R\backslash[-R,R]}|f(x)|^2\e^{\beta|x|}\d x\allowbreak<\epsilon$. So we get for the first integral (independent of $b,b_0$)
\begin{align*}
 \int\displaylimits_{\R\backslash[-R,R]}\!\!|f(x)|^2(\e^{b_0x}-\e^{bx})^2\d x &\le \int\displaylimits_{\R\backslash[-R,R]}\!\! |f(x)|^2\e^{2|x|\max\{|b|,|b_0|\}}\d x\\
      &\le \int\displaylimits_{\R\backslash[-R,R]}\!\! |f(x)|^2\e^{\beta|x|}\d x<\epsilon.
\end{align*}
The second integral in \eqref{1_2} converges to zero, for any fixed $R>0$, as $b\to b_0$. Altogether
\[\lim_{b\to b_0} \|f(x)\e^{bx}-f(x)\e^{b_0x}\|_{L^2(\R)}^2<\epsilon.\]
\end{proof}

\begin{proof}[\underline{Proof of Lemma \ref{abschluss_LL}}]
According to Corollary \ref{cor_dissip} the operator $(L-1-\beta^2/2)|_{C_0^\infty(\R)}$ is dissipative, so it is closable (cf.~\cite[Theorem 1.4.5 (c)]{pazy}), and so is $L|_{C_0^\infty(\R)}$. We define $\L:=\cl_\E L|_{C_0^\infty(\R)}$, and the domain $D(\L)$ consists of all $f\in\E$ such that there exists some $h\in\E$ such that (for some $(f_n)_{n\in\N}\subset C_0^\infty(\R)$) 
\[
 \begin{cases}
   \lim_{n\to\infty}\|f_n-f\|_\omega=0,\\
   \lim_{n\to\infty}\|Lf_n-h\|_\omega=0.
 \end{cases}
\]
For such $f$ we have $\L f:= h=\mathfrak L f$. Therefore $D(\L)\subseteq\{f\in \E:\mathfrak Lf\in\E\}$. Since $\|\cdot\|_E$ is stronger than $\|\cdot\|_\omega$ we also have $D(L)\subset D(\L)$.

Finally we need to show that the above inclusion for the domain indeed is an equality. We take $\zeta\in\C$ with $\Re\zeta\ge 1+\beta^2/2$. From Theorem \ref{prop:fp_in_H} and the dissipativity of $\zeta-\L$ we know that $(\zeta-\L)^{-1}|_E=(\zeta-L)^{-1}$ is a well-defined operator on $E$. And from \eqref{comp_est} we conclude that this is even a  bounded operator in $\E$ with dense domain $E$. Therefore, also its closure $\cl_\E( (\zeta-\L)^{-1}|_E)=(\zeta-\L)^{-1}$ is bounded in $\E$, and therefore $\zeta\in\rho(\L)$. Now assume that there is some $f\in\E\backslash D(\L)$ such that $\mathfrak Lf\in\E$. Because $\zeta\in\rho(\L)$, $\zeta-\L:D(\L)\to\E$ is a bijection, and therefore there exists a unique $\mathfrak f\in D(\L)$ with $(\zeta-\L)\mathfrak f=(\zeta-\mathfrak L)f$, which is equivalent to the existence of $\mathfrak f^\star\in\E$ with $\mathfrak f^\star\neq 0$ such that $(\zeta-\mathfrak L)\mathfrak f^\star=0$. But according to \eqref{comp_est} this is impossible.
\end{proof}

\begin{lem}\label{lem:proj}
Consider two Hilbert spaces $X\inj\X$, and a projection $\Rho_\X\in\BB(\X)$, such that $\Rho_X:=\Rho_\X|_X\in\BB(X)$. Then $\ran\Rho_\X=\cl_\X\ran\Rho_X$ and $\ker\Rho_\X=\cl_\X\ker\Rho_X$.
\end{lem}

\begin{proof}
We give here the proof of the equality of the ranges, the other identity can be shown analogously, using the complementary projections instead. On the one hand we have $\ran\Rho_X\subseteq\ran\Rho_\X$, and so $\cl_\X\ran\Rho_X\subseteq \ran \Rho_\X$, since $\ran\Rho_\X$ is closed in $\X$ due to the boundedness of $\Rho_\X$. On the other hand $\Rho_\X=\cl_\X\Rho_X$, which implies $\ran\Rho_\X\subseteq\cl_\X\ran\Rho_X$. 
\end{proof}

\begin{lem}\label{lem:funct}
Let $X\inj\X$ be Hilbert spaces, and $\psi_0,\ldots,\psi_{k-1}\in\BB(\X,\C),\,k\in\N$, be linearly independent functionals. Then $\tilde \psi_j:=\psi_j|_X\in\BB(X,\C)$ for all $0\le j\le k-1$, and \[\bigcap_{j=0}^{k-1}\ker\psi_j=\cl_\X\bigcap_{j=0}^{k-1}\ker\tilde\psi_j.\]
\end{lem}

\begin{proof}
The boundedness of the $\tilde \psi_j$ is an immediate consequence of $X\inj\X$. In order to show the second statement, we notice that according to the Riesz representation theorem there exists a unique $x_j\in X$ such that $\tilde\psi_j(\cdot)=\la\cdot,x_j\ra_X$ for every $0\le j\le k-1$, where $\la\cdot,\cdot\ra_X$ denotes the inner product in $X$. The set $\{x_0,\ldots,x_{k-1}\}$ is linearly independent, because the corresponding functionals are.  We now apply the Gram-Schmidt process to $\{x_0,\ldots,x_{k-1}\}$ to obtain the orthonormal family $\{\hat x_0,\ldots,\hat x_{k-1}\}$ with same linear hull. As a consequence, there exists a regular matrix $\Lambda:=(\lambda_\ell^j)_{\ell,j}\in\C^{k\times k}$ such that
$\hat x_\ell=\sum_{j=0}^{k-1}\lambda_\ell^jx_j$.
With this we get
\[
\hat x_\ell\la\cdot,\hat x_\ell\ra_X=\sum_{i,j=0}^{k-1}\lambda_\ell^i\bar\lambda_\ell^j x_i\la \cdot,x_j\ra_X=\sum_{i,j=0}^{k-1}\lambda_\ell^i\bar\lambda_\ell^j x_i\tilde\psi_j(\cdot),\quad 0\le \ell\le k-1.
\]
We may now define the orthogonal projection
\begin{equation}\label{wasweis}
 \Rho_X:=\sum_{\ell=0}^{k-1}\hat x_\ell\la\cdot,\hat x_\ell\ra_X=\sum_{i,j,\ell=0}^{k-1}\lambda_\ell^i\bar\lambda_\ell^j x_i\tilde\psi_j(\cdot).
\end{equation}
It can naturally be extended to a projection $\Rho_\X$ in $\X$ by replacing the $\tilde \psi_j$ by $\psi_j$. Since $\psi_j\in\BB(\X,\C)$ for all $0\le j\le k-1$, there follows $\Rho_\X\in\BB(\X)$ from (\ref{wasweis}). Now we apply Lemma \ref{lem:proj} to $\Rho_X\subset\Rho_\X$ to obtain $\ker\Rho_\X=\cl_\X\ker\Rho_X$.

Now it remains to characterize the kernels of the projections. Due to (\ref{wasweis}) we have $\Rho_X f=0$ in $X$ iff 
\begin{equation}\label{sys_eq:psi}
 \sum_{j=0}^{k-1}\tilde\psi_j(f)\sum_{\ell=0}^{k-1}\lambda_\ell^i\bar\lambda_\ell^j=0,\quad 0\le i\le k-1,
\end{equation}
since the vectors $x_i$ are linearly independent. We note that the sums $\sum_{\ell=0}^{k-1}\lambda_\ell^i\bar\lambda_\ell^j$ for $0\le i,j\le k-1$ are the elements of the matrix $\Lambda_2:=\Lambda\Lambda^*,$ where $\Lambda^*$ is the Hermitian conjugate of $\Lambda$. Since $\Lambda_2$ is regular, it follows that (\ref{sys_eq:psi}) holds iff $\tilde\psi_j(f)=0$ for all $0\le j\le k-1.$ The proof of $\Rho_\X f=0$ iff $\psi_j(f)=0$ for all $0\le j\le k-1$ is analogous.
\end{proof}

\begin{proof}[\underline{Proof of Lemma \ref{pr_pc:d}}]
For this we use the norm $\nn\cdot\nn_\omega$. We compute
\begin{align*}
\nn\nabla f\nn_\omega^2 &=\sum_{j=1}^d\sum_{\ell=1}^d\Big( \textstyle \big\|\big(\xi_j+\ii\frac\beta 2 \delta_{j\ell}\big)\hat f\big(\bfxi+\ii\frac\beta2{\mathbf e}_\ell\big)\big\|^2_{L^2(\R^d_{\bfxi})}\\
&\qquad\qquad\qquad
\textstyle+\big\|\big(\xi_j-\ii\frac\beta 2 \delta_{j\ell}\big)\hat f\big(\bfxi-\ii\frac\beta2{\mathbf e}_\ell\big)\big\|^2_{L^2(\R^d_{\bfxi})}\Big)\\
&\ge \sum_{\ell=1}^d\Big( \textstyle \big\|\big(\xi_\ell+\ii\frac\beta 2\big)\hat f\big(\bfxi+\ii\frac\beta2{\mathbf e}_\ell\big)\big\|^2_{L^2(\R^d_{\bfxi})}+\big\|\big(\xi_\ell-\ii\frac\beta 2 \big)\hat f\big(\bfxi-\ii\frac\beta2{\mathbf e}_\ell\big)\big\|^2_{L^2(\R^d_{\bfxi})}\Big)\\
&\ge\big( {\textstyle{\frac\beta 2}}\big)^2\sum_{\ell=1}^d\Big( \textstyle \big\|\hat f\big(\bfxi+\ii\frac\beta2{\mathbf e}_\ell\big)\big\|^2_{L^2(\R^d_{\bfxi})}+\big\|\hat f\big(\bfxi-\ii\frac\beta2{\mathbf e}_\ell\big)\big\|^2_{L^2(\R^d_{\bfxi})}\Big)= \big( {\textstyle{\frac\beta 2}}\big)^2 \nn f\nn_\omega^2.
\end{align*}
This proves the Poincar\'e inequality with the constant $C_p=\frac \beta 2$.
\end{proof}

\bigskip
\footnotesize{{\bf Acknowledgement.} The authors were supported by the FWF (project I 395-N16 and the doctoral school ``Dissipation and dispersion in non-linear partial differential equations'') and the \"OAD-project ``Long-time asymptotics for evolution equations in chemistry and biology''.
}

\end{document}